\documentclass[12pt]{article}
\usepackage{amsfonts}
\usepackage{eufrak}
\usepackage{dsfont}

\usepackage{amsmath}
\usepackage{amstext}
\usepackage{amsopn}
\usepackage{amsbsy}
\usepackage{amscd}
\usepackage{amsxtra}
\usepackage{amsthm}
\usepackage{amssymb}

\usepackage{color}

\newcommand{\ds}{\displaystyle}

\newcommand{\BE}{\begin{equation}}
\newcommand{\BEN}{\begin{equation*}}
\newcommand{\EE}{\end{equation}}
\newcommand{\EEN}{\end{equation*}}
\newcommand{\BL}{\begin{lemma}}
\newcommand{\EL}{\end{lemma}}
\newcommand{\BT}{\begin{theorem}}
\newcommand{\ET}{\end{theorem}}
\newcommand{\BP}{\begin{proposition}}
\newcommand{\EP}{\end{proposition}}
\newcommand{\BC}{\begin{corollary}}
\newcommand{\EC}{\end{corollary}}
\newcommand{\BR}{\begin{remark}}
\newcommand{\ER}{\end{remark}}

\numberwithin{equation}{section}

\title{Exponential decay estimates for the stability of boundary layer solutions to Poisson-Nernst-Planck systems: one spatial dimension case}

\date{}
\author{Chia-Yu Hsieh  \thanks{Department of Mathematics, National Taiwan University, Taipei, Taiwan 10617, email: {\tt b92201049@gmail.com}},
\and Tai-Chia Lin \thanks{Institute of Applied Mathematical Sciences, Center for Advanced Study in Theoretical Sciences (CASTS), National Taiwan University, Taipei, Taiwan 10617,  email:{\tt  tclin@math.ntu.edu.tw}}
}

\newtheorem{thm}{Theorem}[section]

\newtheorem{prop}[thm]{Proposition}
\newtheorem{cor}[thm]{Corollary}

\newtheorem{theorem}{Theorem}[section]
\newtheorem{lemma}[theorem]{Lemma}
\newtheorem{proposition}[theorem]{Proposition}
\newtheorem{corollary}[theorem]{Corollary}
\newtheorem{remark}[theorem]{Remark}

\begin{document}
\maketitle

\begin{abstract}
With a small parameter $\varepsilon$, Poisson-Nernst-Planck (PNP) systems over a finite one-dimensional (1D) spatial domain have steady state solutions, called 1D boundary layer solutions, which profiles form boundary layers near boundary points and become flat in the interior domain as $\varepsilon$ approaches zero. For the stability of 1D boundary layer solutions to (time-dependent) PNP systems, we estimate the solution of the perturbed problem with global electroneutrality. We prove that the $H^{-1}_x$ norm of the solution of the perturbed problem decays exponentially (in time) with exponent independent of $\varepsilon$ if the coefficient of the Robin boundary condition of electrostatic potential has a suitable positive lower bound. The main difficulty is that the gradients of 1D boundary layer solutions at boundary points may blow up as $\varepsilon$ tends to zero. The main idea of our argument is to transform the perturbed problem into another parabolic system with a new and useful energy law for the proof of the exponential decay estimate.
\end{abstract}

\section{Introduction}
\ \ \ \ The Poisson-Nernst-Planck (PNP) system, a well-known mathematical model for ion transport, plays a crucial role in the study of many physical and biological problems~\cite{AMT-ttsp00,BD-ahp00,CLB-bp97,E98,G-zamm85,H1,MRS,Mori1,PJ,RCA89,Rolf2}. Such a model can be represented as
\begin{eqnarray}
   {{n}_{t}} &=& -\nabla \cdot {{J}_{n}}\,,\hspace*{3.1cm}\quad {{p}_{t}}=-\nabla \cdot {{J}_{p}}\,, \label{fx2-o} \\
   {{J}_{n}} &=& -{{D}_{n}}\left( \nabla n-\frac{{{z}_{n}}e}{{{k}_{B}}T}n\nabla \phi  \right)\,,\quad {{J}_{p}}= -{{D}_{p}}\left( \nabla p+\frac{{{z}_{p}}e}{{{k}_{B}}T}p\nabla \phi  \right),  \label{fx1-o} \\
   {{\varepsilon}}\Delta \phi &=& -\rho +{{z}_{n}}en-{{z}_{p}}ep\,,   \label{po1-o}
\end{eqnarray}
for $x\in \Omega, t>0$, where $\left( n,p,\phi \right)$ depends on $x$ and $t$, $\Omega \subset {{\mathbb{R}}^{N}}$ is a bounded smooth domain in ${{\mathbb{R}}^{N}},N\ge 1$, $\nabla=\left( {{\partial }_{{{x}_{1}}}},\cdots ,{{\partial }_{{{x}_{N}}}} \right)$ and $\Delta =\sum\limits_{j=1}^{N}{\partial _{{{x}_{j}}}^{2}}$ is the Laplacian.
Physically, $\phi$ is the electrostatic potential, $n$ is the charge density of anions, $p$ is the charge density of cations, $\rho$ is the permanent (fixed) charge density in the domain, $z_n$, $z_p$ are the valence of ions, $e$ is the elementary charge, $k_B$ is the Boltzmann constant, $T$ is temperature, $J_n, J_p$ are the ionic flux densities and $D_n, D_p$ are their diffusion coefficients. The parameter $\varepsilon$ related to the dielectric constant and the Debye length can be assumed as a small parameter tending to zero (cf.~\cite{BCE, EL-sm07,Lu}). For simplicity, we only consider monovalent ions, that is, $z_n=z_p=1$, and set $e/k_BT=1$, $\rho=0$, $D_n=D_p=1$. Besides, we rescale $\varepsilon$ and transform (\ref{fx2-o})-(\ref{po1-o}) into
\begin{align}
 n_t & = -\nabla\cdot J_n\,, \hspace{1.7cm} p_t = -\nabla\cdot J_p, \label{fx2}\\
 J_n & = -(\nabla n -n \nabla \phi)\,, \quad J_p = -(\nabla p +p \nabla \phi), \label{fx1}\\
 \varepsilon \Delta \phi & = n -p\,,  \label{po1}
\end{align}
for $x\in \Omega, t>0$.

Debye (diffuse) layers occur in ionic liquids near electrodes and have many applications in the fields of chemical physics and biophysics (cf.~\cite{H-Z-81}). To see Debye layers, solutions of (\ref{fx2})-(\ref{po1}) with boundary layers need to be investigated. For simplicity, the domain $\Omega$ is set as $\Omega =\left( -1,1 \right)$ a one-dimensional interval in the whole paper. Then (\ref{fx2})-(\ref{po1}) can be denoted as
\begin{align}\label{id1}
\begin{cases}
n_t = \partial_x (n_x - n \phi_x)\,, \\
p_t = \partial_x (p_x + p \phi_x)\,, \\
\varepsilon \phi_{xx} = n - p \,,
\end{cases}
\end{align}
for $x\in (-1,1), t>0$. For the boundary conditions of (\ref{id1}), we consider no-flux boundary conditions of $n$ and $p$ to describe the insulated domain boundaries, which are commonly used to study physical (biophysical) phenomena like  the electric double layer and the ion transport through channels. Besides, we use Robin type boundary condition of $\phi$ to represent the capacitance effect of physical systems (cf.~\cite{LMB, Mori1, Rolf2}) given as follows:
\begin{align}\label{id2}
\begin{cases}
n_x - n \phi_x = p_x + p \phi_x = 0,\quad \mbox{at} \quad x = \pm 1\,, \\
\phi + \gamma_\varepsilon \phi_x = \phi_0 (1) \hspace{2cm} \mbox{at} \quad x = 1\,, \\
\phi - \gamma_\varepsilon \phi_x = \phi_0 (-1) \hspace{1.7cm} \mbox{at} \quad x = -1\,,\\
\end{cases}
\end{align}
where $\phi_0 (1)$, $\phi_0 (-1)$ are constants, and $\gamma_\varepsilon> 0$ is a constant depending on $\varepsilon$. System (\ref{id1}) with (\ref{id2}) has the conservation of total charges of the individual ions $\int_{-1}^{1}{n\,dx}=A$, $\int_{-1}^{1}{p\,dx}=B$ for $t>0$, where $A$ and $B$ are positive constants (independent of $t$) representing total negative and positive charges, respectively. In most of the physical and biological systems, global electroneutrality holds true which means the total positive charge equal to the total negative charge. Consequently, we assume that $A=B=m_0>0$, i.e., global electroneutrality holds true in the whole paper.

System (\ref{id1}) with (\ref{id2}) has a steady state solution $\left( n,p,\phi  \right)=\left( {{n}^{0}},{{p}^{0}},\psi\right)$ denoted as
\begin{align}\label{id4-np0}
n^0 = \frac{m_0\,e^{\psi}}{\int_{-1}^{1} e^{\psi} dx}\,, \quad p^0 = \frac{m_0\,e^{-\psi}}{\int_{-1}^{1} e^{-\psi} dx}\,,
\end{align}
where $m_0>0$ is a constant and $\psi$ is the solution of the following equation, called charge-conserving Poisson-Boltzmann equation (cf.~\cite{WXLLS-N-14}), with Robin type boundary conditions:
\begin{align}\label{id3-p}
\begin{cases}
\varepsilon \psi_{xx} = \ds\frac{m_0\,e^{\psi}}{\int_{-1}^{1} e^{\psi} dx} - \ds\frac{m_0\,e^{-\psi}}{\int_{-1}^{1} e^{-\psi} dx} \quad \mbox{in} \quad (-1,1), \\
\\
\psi(\pm 1) \pm \gamma_\varepsilon \psi_x (\pm 1) = \phi_0 (\pm 1).
\end{cases}
\end{align}
Note that $n_{x}^{0}-{{n}^{0}}{{\psi }_{x}}=p_{x}^{0}+{{p}^{0}}{{\psi }_{x}}=0$ and $\varepsilon {{\psi }_{xx}}={{n}^{0}}-{{p}^{0}}$. Without loss of generality, we may assume ${{\phi }_{0}}(1)=-{{\phi }_{0}}(-1)>0$. From~\cite{LLHLL}, we get the following results of boundary layer solutions of (\ref{id3-p}):

\noindent {\bf Theorem~A.} (cf.~\cite{LLHLL}) Let $\psi$ be the solution of (\ref{id3-p}). Then $\psi$ is odd, i.e., $\psi(x)=-\psi(-x)$ for $x\in [-1,1]$, increasing in $(-1,1)$, convex in $(0,1)$, and concave in $(-1,0)$. Moreover, $\psi$ satisfies
\begin{enumerate} \item[(i)]~~~Interior Estimate:
$$\left| \psi \left( x \right) \right|\le {{\phi }_{0}}\left( 1 \right)\left( {{e}^{-\frac{M}{\sqrt{\varepsilon}}\left( 1+x \right)}}+{{e}^{-\frac{M}{\sqrt{\varepsilon}}\left( 1-x \right)}} \right)\text{ and }\underset{\varepsilon \to 0+}{\mathop{\lim }}\,\psi \left( x \right)=0\text{  for  }x\in \left( -1,1 \right)\,,$$
where $M=\sqrt{\frac{{{m}_{0}}}{2{{e}^{{{\phi }_{0}}\left( 1 \right)}}}}$ is a positive constant independent of $\varepsilon$.
\item[(ii)]~~Boundary Estimate: \\
If $0\leq\underset{\varepsilon \to 0+}{\mathop{\lim }}\,\frac{{{\gamma }_{\varepsilon }}}{\sqrt{\varepsilon }}=\gamma<\infty$, then $\underset{\varepsilon \to 0+}{\mathop{\lim }}\,\psi \left( 1 \right)={{\psi }^{*}}$ and $\underset{\varepsilon \to 0+}{\mathop{\lim }}\,\sqrt{\varepsilon }{\psi }'\left( 1 \right)=\sqrt{\alpha }\left( {{e}^{{}^{{{\psi }^{*}}}\!\!\diagup\!\!{}_{2}\;}}-{{e}^{{}^{-{{\psi }^{*}}}\!\!\diagup\!\!{}_{2}\;}} \right)$, \\ where $0<{{\psi }^{*}}\le {{\phi }_{0}}\left( 1 \right)$ is uniquely determined by ${{\phi }_{0}}\left( 1 \right)-{{\psi }^{*}}=\gamma \sqrt{\alpha }\left( {{e}^{{}^{{{\psi }^{*}}}\!\!\diagup\!\!{}_{2}\;}}-{{e}^{{}^{-{{\psi }^{*}}}\!\!\diagup\!\!{}_{2}\;}} \right)$. \\
\end{enumerate}
Theorem~A implies that as $\varepsilon$ goes to zero, $\psi$ have asymptotic behavior of boundary layer so we call $\left( n,p,\phi\right)=\left( {{n}^{0}},{{p}^{0}},\psi\right)$ as a boundary layer solution of system (\ref{id1}) with (\ref{id2}). Note that $({{n}^{0}},{{p}^{0}})$ is represented in (\ref{id4-np0}).

To get the stability of the boundary layer solution $\left( {{n}^{0}},{{p}^{0}},\psi\right)$ to system (\ref{id1}) with (\ref{id2}), we study the perturbed problem~(\ref{id9}) with~(\ref{id10}) which comes from the assumption that system (\ref{id1}) with (\ref{id2}) has solution
$$\left( n,p,\phi  \right)=\left( {{n}^{0}},{{p}^{0}},\psi \right)+\left( \tilde{n},\tilde{p},\tilde{\phi } \right)\,, $$
where $\left( \tilde{n},\tilde{p},\tilde{\phi } \right)$ satisfies the perturbed problem
\begin{align}\label{id9}
\begin{cases}
\tilde{n}_t = \tilde{n}_{xx} - (n^0 \tilde{\phi}_x)_x - (\tilde{n} \psi_x)_x - (\tilde{n} \tilde{\phi}_x)_x, \\
\tilde{p}_t = \tilde{p}_{xx} + (p^0 \tilde{\phi}_x)_x + (\tilde{p} \psi_x)_x + (\tilde{p} \tilde{\phi}_x)_x, \\
\varepsilon \tilde{\phi}_{xx} = \tilde{n} - \tilde{p}\,,\quad\hbox{ for }\quad x\in (-1,1)\,,\: t>0\,,
\end{cases}
\end{align}
with boundary conditions:
\begin{align}\label{id10}
\begin{cases}
\tilde{n}_x - n^0 \tilde{\phi}_x - \tilde{n} \psi_x - \tilde{n} \tilde{\phi}_x = \tilde{p}_x + p^0 \tilde{\phi}_x + \tilde{p} \psi_x + \tilde{p} \tilde{\phi}_x = 0 \quad \mbox{at} \quad x = \pm 1, \\
\tilde{\phi} \pm \gamma_\varepsilon \tilde{\phi}_x = 0 \quad \mbox{at} \quad x = \pm 1\,.
\end{cases}
\end{align}
Here $\tilde{n}$ and $\tilde{p}$ denote charge density perturbation of anions and cations, respectively. To let the global electroneutrality hold true, i.e., $\int_{-1}^{1}{ndx=A=B=\int_{-1}^{1}{pdx}}>0$, we need to have $\int_{-1}^{1}{\tilde{n}dx=\int_{-1}^{1}{\tilde{p}dx}}=0$, which comes from condition (\ref{id20}). Otherwise, if $A\neq B$ and $A,B>0$ independent of $\varepsilon$, then as $\varepsilon$ approaches zero, steady state solution of (\ref{id1}) with (\ref{id2}) becomes unbounded and far away from $\left( {{n}^{0}},{{p}^{0}},\psi \right)$ the boundary layer solution of (\ref{id1}) with (\ref{id2}) for the case of $A=B=m_0>0$ (cf.~\cite{LLHLL}). It seems impossible to get the stability of $\left( {{n}^{0}},{{p}^{0}},\psi \right)$ if condition (\ref{id20}) fails and $\int_{-1}^{1}{ndx=A\ne B=\int_{-1}^{1}{pdx}}$ ($A,B>0$ independent of $\varepsilon$) holds true. This motivates us to assume $A=B=m_0>0$ and (\ref{id20}) in the whole paper.

Conventionally, the stability of (\ref{id1}) with $\varepsilon=1$ and the Dirichlet boundary condition for $\phi$ holds true because of $\lim_{t\rightarrow\infty} \|\tilde{n}\|_{L^\infty_x} + \|\tilde{p}\|_{L^\infty_x} = 0$ (cf.~\cite{bhn-na94}) and the exponential decay estimate $\|\tilde{n}\|_{L^1_x} + \|\tilde{p}\|_{L^1_x} \leq Ce^{-\lambda t}$ (cf.~\cite{BD-ahp00}) for $t>0$, where the constant $C$ and exponent $\lambda$ are positive. Here we study the stability of (\ref{id1}) with $0<\varepsilon\ll 1$ a small parameter tending to zero and the Robin boundary condition for $\phi$ (see~(\ref{id2})). It seems difficult to get the exponential decay estimate in $L^1_x$-norm with exponent independent of $\varepsilon$. The main difficulty is that the profile of the solution $\psi$ has boundary layers near boundary points $x=\pm 1$, and $\psi_x$ blows up at boundary points $x=\pm 1$ with order $\varepsilon^{-1/2}$, i.e., $|\psi_x(\pm 1)|\sim \varepsilon^{-1/2}$ as $\varepsilon$ tends to zero (cf.~\cite{LLHLL}). Instead of the $L^1_x$-norm estimate, we prove the exponential decay estimate in $H^{-1}_x$-norm denoted as $\left\| {\tilde{n}} \right\|_{H_{x}^{-1}}^{2}+\left\| {\tilde{p}} \right\|_{H_{x}^{-1}}^{2}\le {{I}_{0}}{{e}^{-\alpha t}}$ for $t>0$, where $\alpha$ is a positive constant independent of $\varepsilon$, and $I_0$ is a positive constant depending on the $H^{-1}_x$-norm of initial data ${{\left. {\tilde{n}} \right|}_{t=0}}$ and ${{\left. {\tilde{p}} \right|}_{t=0}}$.

\subsection{Main Results}
\ \ \ \
To study system (\ref{id9}) with (\ref{id10}), we introduce the change of variables
\begin{align}\label{id11}
\tilde{\delta} = \tilde{n} - \tilde{p},\quad \tilde{\eta} = \tilde{n} + \tilde{p}\,,
\end{align}
where $\tilde{\delta}$ is the gap between concentrations of positive (cations) and negative (anions) ions. Note that  $\int_{-1}^{1}{\tilde{\delta }dx=0}$, i.e., $\int_{-1}^{1}{\tilde{n}dx=\int_{-1}^{1}{\tilde{p}dx}}$ means the total positive charge equal to the total negative charge, and then global electroneutrality holds true.
By (\ref{id11}), system (\ref{id9}) with (\ref{id10}) becomes
\begin{align}\label{id12}
\begin{cases}
\tilde{\delta}_t = \tilde{\delta}_{xx} - (\eta^0 \tilde{\phi}_x)_x - (\tilde{\eta} \psi_x)_x - (\tilde{\eta} \tilde{\phi}_x)_x, \\
\tilde{\eta}_t = \tilde{\eta}_{xx} - (\delta^0 \tilde{\phi}_x)_x - (\tilde{\delta} \psi_x)_x - (\tilde{\delta} \tilde{\phi}_x)_x, \\
\varepsilon \tilde{\phi}_{xx} = \tilde{\delta},
\end{cases}
\end{align}
with boundary conditions:
\begin{align}\label{id13}
\begin{cases}
\tilde{\delta}_x - \eta^0 \tilde{\phi}_x - \tilde{\eta} \psi_x - \tilde{\eta} \tilde{\phi}_x = \tilde{\eta}_x - \delta^0 \tilde{\phi}_x - \tilde{\delta} \psi_x - \tilde{\delta} \tilde{\phi}_x = 0 \quad \mbox{at} \quad x = \pm 1, \\
\tilde{\phi} \pm \gamma_\varepsilon \tilde{\phi}_x = 0 \quad \mbox{at} \quad x = \pm 1,
\end{cases}
\end{align}
where due to (\ref{id11}),
\begin{align}\label{id14}
\delta^0 = n^0 - p^0, \quad \eta^0 = n^0 + p^0\,,
\end{align}
and $\left(n^0,p^0\right)$ is defined in (\ref{id4-np0}). By (\ref{id4-np0}), (\ref{id3-p}) and (\ref{id14}), $\psi$ satisfies
\BE\label{pd1}
\varepsilon {{\psi}_{xx}}={{\delta }^{0}}\quad\hbox{ for }\quad x\in (-1,1)\,.
\EE

\noindent {\bf Linear Stability}
\medskip

To get linear stability, we consider the linearized problem of (\ref{id12}) with (\ref{id13}) as follows:
\begin{align}\label{id15}
\begin{cases}
\tilde{\delta}_t = \tilde{\delta}_{xx} - (\eta^0 \tilde{\phi}_x)_x - (\tilde{\eta} \psi_x)_x, \\
\tilde{\eta}_t = \tilde{\eta}_{xx} - (\delta^0 \tilde{\phi}_x)_x - (\tilde{\delta} \psi_x)_x, \\
\varepsilon \tilde{\phi}_{xx} = \tilde{\delta},
\end{cases}
\end{align}
with boundary conditions:
\begin{align}\label{id16}
\begin{cases}
\tilde{\delta}_x - \eta^0 \tilde{\phi}_x - \tilde{\eta} \psi_x = \tilde{\eta}_x - \delta^0 \tilde{\phi}_x - \tilde{\delta} \psi_x  = 0 \quad \mbox{at} \quad x = \pm 1, \\
\tilde{\phi} \pm \gamma_\varepsilon \tilde{\phi}_x = 0 \quad \mbox{at} \quad x = \pm 1\,.
\end{cases}
\end{align}

By (\ref{id15}) and (\ref{id16}), it is obvious that
\begin{align}\nonumber 
\frac{d}{dt} \int_{-1}^{1} \tilde{\delta} dx = \frac{d}{dt} \int_{-1}^{1} \tilde{\eta} dx = 0 \quad \mbox{for} \quad t > 0\,,
\end{align}
and then
\begin{align}
\int_{-1}^{1} \tilde{\delta} (x,t) dx &= \int_{-1}^{1} \tilde{\delta}_0(x) dx, \label{id18}\\
\int_{-1}^{1} \tilde{\eta} (x,t) dx &= \int_{-1}^{1} \tilde{\eta}_0(x) dx\,, \label{id19}
\end{align}
for all $t>0$, where $\tilde{\delta}_0(x)=\tilde{\delta}(x,0)$ and $\tilde{\eta}_0(x)=\tilde{\eta} (x,0)$ are the initial data. Let
\begin{align}\label{id21}
D(x,t) = \int_{-1}^{x} \tilde{\delta}(s,t) ds
\end{align}
and
\begin{align}\label{id22}
H(x,t) = \int_{-1}^{x} \tilde{\eta}(s,t) ds\,,
\end{align}
for $x\in (-1,1)$ and $t>0$. Then by (\ref{id18})-(\ref{id22}) the boundary conditions of $D$ and $H$ at $x=\pm 1$ become the zero Dirichlet boundary condition
\BE\label{0-DBC}
D=H=0\quad\hbox{ at }\quad x=\pm 1\,,\quad\hbox{ for }\quad t>0\,,
\EE
if the initial data $\tilde{\delta}_0(x)=\tilde{\delta}(x,0)$ and $\tilde{\eta}_0(x)=\tilde{\eta} (x,0)$ satisfy
\begin{align}
\int_{-1}^{1} \tilde{\delta}_0(x) dx = \int_{-1}^{1} \tilde{\eta}_0(x) dx = 0\,. \label{id20}
\end{align}
The physical meaning of (\ref{id20}) is: the total charge of solution $(n,p)=({{n}^{0}}+\tilde{n},{{p}^{0}}+\tilde{p})$ is same as that of $\left( {{n}^{0}},{{p}^{0}} \right)$ so the (global) electroneutrality $\int_{-1}^{1}{n\left( x,t \right)dx}=\int_{-1}^{1}{{{n}^{0}}\left( x \right)dx}=\int_{-1}^{1}{{{p}^{0}}\left( x \right)dx}=\int_{-1}^{1}{p\left( x,t \right)dx}$ holds true for $t>0$.


We use (\ref{id21}), (\ref{id22}), and integrate equations of (\ref{id15}) from $-1$ to $x$. Then we get
\begin{align}
D_t &= D_{xx} - \eta^0 \tilde{\phi}_x - \psi_x H_x\,, \label{id24} \\
H_t &= H_{xx} - \delta^0 \tilde{\phi}_x - \psi_x D_x\,, \label{id25}
\end{align}
which give the energy law of $(D,H)$ expressed as follows:
\begin{align}\label{eng-lw1}
\frac{1}{2} \frac{d}{dt} \int_{-1}^{1} (D^2 + H^2) dx &= -\int_{-1}^1 \Big(D_x^2 + H_x^2 + \frac{1}{\varepsilon} \eta^0 D^2\Big) dx - \tilde{\phi}_x (-1,t) \int_{-1}^1 (\eta^0 D + \delta^0 H) dx
\end{align}
for $t>0$ (see Theorem~\ref{thm1}). However, $\tilde{\phi}_x (-1,t) = -\frac{1}{2(1+\gamma_\varepsilon)\varepsilon} \int_{-1}^1 D dx$ (see (\ref{id37}) in Section~\ref{lzp}) and ${{\gamma }_{\varepsilon }}\sim \sqrt{\varepsilon }\gamma $ as $\varepsilon \to 0+$ (see Theorem~A) imply that $\left|\tilde{\phi}_x (-1,t)\right|$ becomes extremely large as $\varepsilon$ approaches zero if the integral $\int_{-1}^1 D dx$ is away from zero. This makes (\ref{eng-lw1}) hard to be used for the $L^2_x$ estimates of $D$ and $H$.

To overcome such difficulty, we use the following transformation:
\begin{eqnarray*}
\bar{D}(x,t) &=& D(x,t) - \frac{1}{2} \int_{-1}^1 D(x,t) dx\,, \\
\bar{H}(x,t) &=& H(x,t) - \frac{1}{2} \int_{-1}^1 H(x,t) dx\,,
\end{eqnarray*}
which can be denoted as $\bar{D}(x,t) = D(x,t) - d(t)$ and $\bar{H}(x,t) = H(x,t) - h(t)$,
where
\begin{align}
d(t) = \frac{1}{2} \int_{-1}^1 D(x,t) dx\,, \notag \\
h(t) = \frac{1}{2} \int_{-1}^1 H(x,t) dx\,. \notag
\end{align}
Then we prove the following result for (\ref{id24}) and (\ref{id25}).
\begin{thm}\label{prop1}
Suppose (\ref{id20}), $0 < \gamma_\varepsilon \leq \gamma_{max} < \infty$, and $\frac{\gamma_\varepsilon}{\sqrt{\varepsilon}} > \frac{\left((1+2\gamma_{max})^2+3\right)K_0\phi_0(1)}{4M}$ hold true, where $M=\sqrt{\frac{{{m}_{0}}}{2{{e}^{{{\phi }_{0}}\left( 1 \right)}}}}>0$, $K_0=\underset{0<\left| y \right|\le {{\phi }_{0}}\left( 1 \right)}{\mathop{\sup }}\,\frac{\left| {{e}^{y}}-1 \right|}{\left| y \right|}>0$ and $\gamma_{max}>0$ is a constant independent of $\varepsilon$. Then
\begin{align}\label{id47}
\frac{d}{dt}\left[\frac{1}{2}\int_{-1}^{1}{({{{\bar{D}}}^{2}}+{{{\bar{H}}}^{2}})}dx+{{d}^{2}}+{{h}^{2}}\right] \le -\int_{-1}^{1}{\left( \bar{D}_{x}^{2}+\frac{1}{2}\bar{H}_{x}^{2} \right)}
-\frac{1}{\varepsilon }{{m}_{0,\varepsilon }}\int_{-1}^{1}{{{{\bar{D}}}^{2}}}\,,
\end{align}
for $t>0$ and $0<\varepsilon<\tilde{\varepsilon }$, where ${{m}_{0,\varepsilon }}= m_0\left[ 1 - K_0\phi_0(1)\frac{\sqrt{\varepsilon}}{M} \right]> \frac{m_0}{2}$, and $\tilde{\varepsilon }$ depends only on $m_0$ and $\phi_0(1)$.
\end{thm}
\begin{remark}
Constant $K_0$ is defined by ${{K}_{0}}=\underset{0<\left| y \right|\le {{\phi }_{0}}\left( 1 \right)}{\mathop{\sup }}\,\frac{\left| {{e}^{y}}-1 \right|}{\left| y \right|}>0$ which depends only on ${{\phi }_{0}}\left( 1 \right)$ and approaches to one as ${{\phi }_{0}}\left( 1 \right)$ tends to zero. Hence for any $\gamma>0$, if $\displaystyle{\lim_{\varepsilon\rightarrow 0+}} \frac{\gamma_\varepsilon}{\sqrt{\varepsilon}} = \gamma$, the hypothesis $\frac{\gamma_\varepsilon}{\sqrt{\varepsilon}} > \frac{\left((1+2\gamma_{max})^2+3\right)K_0\phi_0(1)}{4M}$ can be fulfilled if $\phi_0(1)>0$ is sufficiently small or $m_0>0$ is sufficiently large for all small $\varepsilon$, where $M=\sqrt{\frac{{{m}_{0}}}{2{{e}^{{{\phi }_{0}}\left( 1 \right)}}}}>0$.
\end{remark}

From Theorem~\ref{prop1}, we obtain the following estimates.
\begin{cor}\label{cor1}
If $0 < \gamma_\varepsilon \leq \gamma_{max} < \infty$ and $\frac{\gamma_\varepsilon}{\sqrt{\varepsilon}} > \frac{\left((1+2\gamma_{max})^2+3\right)K_0\phi_0(1)}{4M}$, then there exists a positive constant $\alpha > 0$, independent of $\varepsilon$, such that
\begin{align}\label{id55}
\frac{1}{2}\int_{-1}^1 (\bar{D}^2 + \bar{H}^2) dx + d^2 + h^2 \leq I_0 e^{-\alpha t}
\end{align}
for $t>0$, where
\begin{align}\label{id56}
I_0 = \bigg[ \displaystyle\frac{1}{2}\int_{-1}^1 (\bar{D}^2 + \bar{H}^2) dx + d^2 + h^2 \bigg]_{t=0}.
\end{align}
\end{cor}
\noindent Because of $D=\bar{D}+d$ and $H=\bar{H}+h$, Corollary~\ref{cor1} gives $\left\| D \right\|_{L_{x}^{2}}^{2}+\left\| H \right\|_{L_{x}^{2}}^{2}\le {{I}_{0}}{{e}^{-\alpha t}}$, i.e., $\left\| {\tilde{n}} \right\|_{H_{x}^{-1}}^{2}+\left\| {\tilde{p}} \right\|_{H_{x}^{-1}}^{2}\le {{I}_{0}}{{e}^{-\alpha t}}$ for $t>0$, which implies the linear stability of (\ref{id1}) with (\ref{id2}) in $H^{-1}_x$-norm, where $I_0$ is a positive constant depending only on the initial data. Here we have used the equivalence between ${{\left\| D \right\|}_{L_{x}^{2}}}+{{\left\| H \right\|}_{L_{x}^{2}}}$ and ${{\left\| {\tilde{n}} \right\|}_{H_{x}^{-1}}}+{{\left\| {\tilde{p}} \right\|}_{H_{x}^{-1}}}$ (see Appendix~I).

To get the $H^{-1}_x$-norm estimate, we firstly transform the linear part of the perturbed problem~(\ref{id9}) with~(\ref{id10}) (i.e., (\ref{id11}), (\ref{id21}) and (\ref{id22})) into a coupled system of linear parabolic equations of $(D,H)$ denoted as (\ref{id24}) and (\ref{id25}) with zero Dirichlet boundary condition (\ref{0-DBC}). To preserve the global electroneutrality, we assume that the total charge density perturbation is zero for anions and cations, i.e., the initial data satisfies (\ref{id20}), which implies boundary condition (\ref{0-DBC}). Then we find the associated energy law (\ref{eng-lw1}) (proved in Theorem~\ref{thm1}) but the coefficient of the last integral of (\ref{eng-lw1}) still blow up as $\varepsilon$ tends to zero if the integral $\int_{-1}^{1}{D\,dx}$ is away from zero (see (\ref{id37})). This motivates us to decompose $(D,H)$ into $\left( \bar{D},\bar{H} \right)$ and $(d,h)$, where $\bar{D}=D-d$, $\bar{H}=H-h$, $d=\frac{1}{2}\int_{-1}^{1}{D\,dx}$ and $h=\frac{1}{2}\int_{-1}^{1}{H\,dx}$. Then we derive (\ref{id42}) (see Theorem~\ref{thm2}) as the energy law of $\left( \bar{D},\bar{H}, d, h \right)$ to prove Theorem~\ref{prop1} and Corollary~\ref{cor1} which imply $\left\| D \right\|_{L_{x}^{2}}^{2}+\left\| H \right\|_{L_{x}^{2}}^{2}\le {{I}_{0}}{{e}^{-\alpha t}}$ and hence $\left\| {\tilde{n}} \right\|_{H_{x}^{-1}}^{2}+\left\| {\tilde{p}} \right\|_{H_{x}^{-1}}^{2}\le {{I}_{0}}{{e}^{-\alpha t}}$ for $t>0$, which gives the linear stability of (\ref{id1}) with (\ref{id2}) in $H^{-1}_x$-norm under global electroneutrality. Here $\alpha$ is a positive constant independent of $\varepsilon$ and $I_0$ is a positive constant depending on the $L^2_x$-norm of initial data ${{\left. D \right|}_{t=0}}$ and ${{\left. H \right|}_{t=0}}$, i.e., the $H^{-1}_x$-norm of initial data ${{\left. {\tilde{n}} \right|}_{t=0}}$ and ${{\left. {\tilde{p}} \right|}_{t=0}}$. Note that ${{\left\| D \right\|}_{L_{x}^{2}}}+{{\left\| H \right\|}_{L_{x}^{2}}}$ is equivalent to ${{\left\| {\tilde{n}} \right\|}_{H_{x}^{-1}}}+{{\left\| {\tilde{p}} \right\|}_{H_{x}^{-1}}}$ (see Appendix~I) and $\gamma_\varepsilon>0$ is assumed to have a suitable positive lower bound (see Theorem~\ref{prop1}), which makes the last three terms of (\ref{rv-eng1}) together become nonpositive so (\ref{id47}) holds true. Such an assumption of $\gamma_\varepsilon$ is also used to study nonlinear system (\ref{id12}) with boundary condition (\ref{id13}).
\medskip

\noindent {\bf Nonlinear Stability}
\medskip

For nonlinear stability, we may generalize the idea of linear stability to study $(\tilde{\delta},\tilde{\eta},\tilde{\phi})$ the solution of nonlinear system (\ref{id12}) with boundary condition (\ref{id13}). The main difficulty is to control the nonlinear terms $\tilde{\eta }{{\tilde{\phi }}_{x}}$ and $\tilde{\delta }{{\tilde{\phi }}_{x}}$ of system~(\ref{id12}). Here we assume that the initial data satisfies (\ref{nonlinear-I0}), which implies that the right side of (\ref{nl-stability}) becomes negative (see Theorem~\ref{nl-stability-thm}). Consequently, (\ref{nl-stability}) is useful to show (\ref{cor4-1-ineq}) (see Corollary~\ref{cor4-1}) and get the nonlinear stability of $\left( {{n}^{0}},{{p}^{0}},\psi\right)$ to system (\ref{id1}) with (\ref{id2}).

Now we state results for nonlinear stability as follows:

\begin{thm}\label{nl-stability-thm}
Under the same hypotheses as in Theorem~\ref{prop1}, suppose furthermore that the initial data ${{\left. \left( n,p \right) \right|}_{t=0}}=\left( {{n}_{0}},{{p}_{0}} \right)$ in $(-1,1)$ satisfies ${{n}_{0}},{{p}_{0}}\in {{L}^{2}}\left( -1,1 \right)$ and
$$
n_0(x)={{n}^{0}(x)}+\tilde{n}(x,0)\ge 0\,,\quad p_0(x)={{p}^{0}(x)}+\tilde{p}(x,0)\ge 0 \quad\hbox{ for }\: x\in (-1,1)\,.
$$ Then
\begin{align}\label{nl-stability}
\frac{d}{dt}\left[\frac{1}{2}\int_{-1}^{1}{({{{\bar{D}}}^{2}}+{{{\bar{H}}}^{2}})}dx+{{d}^{2}}+{{h}^{2}}\right] \leq -\int_{-1}^1 \bigg(\bar{D}_x^2 + \left(\frac{1}{2}-\frac{d^2}{m_0 \varepsilon}\right)\bar{H}_x^2\bigg)dx - \frac{1}{4\varepsilon} m_{0,\varepsilon}' \int_{-1}^1 \bar{D}^2 dx
\end{align}
for $t>0$ and $0<\varepsilon<\tilde{\varepsilon}'$, where $m_{0,\varepsilon}' = m_0 \left[1 - 2K_0\phi_0(1)\frac{\sqrt{\varepsilon}}{M}\right] > \frac{m_0}{2}$, and $\tilde{\varepsilon}'$ depends only on $m_0$ and $\phi_0(1)$. Moreover, if the initial data satisfies
\begin{align}\label{nonlinear-I0}
I_0 := \bigg[ \displaystyle\frac{1}{2}\int_{-1}^1 (\bar{D}^2 + \bar{H}^2) dx + d^2 + h^2 \bigg]_{t=0} < \frac{\theta m_0\varepsilon}{2}
\end{align}
for some $0 < \theta < 1$, then $\frac{1}{2}-\frac{d^2}{m_0 \varepsilon} > \frac{1}{2}(1-\theta) > 0$ for all $t \geq 0$.
\end{thm}
\noindent Besides, from Theorem~\ref{nl-stability-thm}, we get
\begin{cor}\label{cor4-1}
Under the same hypotheses of Theorem \ref{nl-stability-thm}, if (\ref{nonlinear-I0}) holds true, then there exists a positive constant $\alpha' > 0$, independent of $\varepsilon$, such that
\begin{align}\label{cor4-1-ineq}
\frac{1}{2}\int_{-1}^1 (\bar{D}^2 + \bar{H}^2) dx + d^2 + h^2 \leq I_0 e^{-\alpha' t}\,,
\end{align}
and
\begin{align}\label{cor4-1-ineq-1}
\left\| D \right\|_{L_{x}^{2}}^{2}+\left\| H \right\|_{L_{x}^{2}}^{2}\le {{I}_{0}}{{e}^{-\alpha' t}}
\end{align}
for $t>0$ and $0<\varepsilon<\tilde{\varepsilon}'$.
\end{cor}
\noindent Due to $D=\bar{D}+d$ and $H=\bar{H}+h$, (\ref{cor4-1-ineq}) may imply (\ref{cor4-1-ineq-1}) and show that the upper bound of $\left\| D \right\|_{L_{x}^{2}}^{2}+\left\| H \right\|_{L_{x}^{2}}^{2}$ being equivalent to~${{\left\| {\tilde{n}} \right\|}_{H_{x}^{-1}}}+{{\left\| {\tilde{p}} \right\|}_{H_{x}^{-1}}}$ (see Appendix~I) approaches zero exponentially with exponent independent of $\varepsilon$ as $t$ goes to infinity. This represents the exponential decay estimate (to~$\varepsilon$) of ${{\left\| {\tilde{n}} \right\|}_{H_{x}^{-1}}}+{{\left\| {\tilde{p}} \right\|}_{H_{x}^{-1}}}$ and gives the nonlinear stability of (\ref{id1}) with (\ref{id2}) in $H^{-1}_x$ norm.

For nonlinear stability, we use the same idea of linear stability to study $(\tilde{\delta},\tilde{\eta},\tilde{\phi})$ the solution of nonlinear system (\ref{id12}) with boundary condition (\ref{id13}). The main difficulty is to control the extra nonlinear terms $\tilde{\eta }{{\tilde{\phi }}_{x}}$ and $\tilde{\delta }{{\tilde{\phi }}_{x}}$ of system~(\ref{id12}). Here we assume that the initial data ${{\left. \left( n,p \right) \right|}_{t=0}}=\left( {{n}_{0}},{{p}_{0}} \right)$ in $(-1,1)$ satisfies ${{n}_{0}},{{p}_{0}}\in {{L}^{2}}\left( -1,1 \right)$, $n_0(x)={{n}^{0}(x)}+\tilde{n}(x,0)\ge 0$, $p_0(x)={{p}^{0}(x)}+\tilde{p}(x,0)\ge 0$ for $x\in (-1,1)$ and (\ref{nonlinear-I0}) which expresses the smallness of ${{\left\| {\tilde{n}} \right\|}_{H_{x}^{-1}}}+{{\left\| {\tilde{p}} \right\|}_{H_{x}^{-1}}}$ at $t=0$. Then we use (\ref{id74}) the energy law of $\left( \bar{D},\bar{H}, d, h \right)$ to show $\left\| D \right\|_{L_{x}^{2}}^{2}+\left\| H \right\|_{L_{x}^{2}}^{2}\le {{I}_{0}}{{e}^{-\alpha^\prime t}}$ for $t>0$ (see Theorem~\ref{nl-stability-thm} and Corollary~\ref{cor4-1}), i.e., $\left\| {\tilde{n}} \right\|_{H_{x}^{-1}}^{2}+\left\| {\tilde{p}} \right\|_{H_{x}^{-1}}^{2}\le {{I}_{0}}{{e}^{-\alpha^\prime t}}$ for $t>0$, where $\alpha^\prime$ is a positive constant independent of $\varepsilon$, and constant $I_0>0$ comes from the $L^2_x$-norm of initial data ${{\left. D \right|}_{t=0}}$ and ${{\left. H \right|}_{t=0}}$, i.e., the $H^{-1}_x$-norm of initial data ${{\left. {\tilde{n}} \right|}_{t=0}}$ and ${{\left. {\tilde{p}} \right|}_{t=0}}$ satisfying (\ref{nonlinear-I0}). Note that condition $n_0(x), p_0(x)\geq 0$ for $x\in (-1,1)$ implies $n(x,t), p(x,t)\geq 0$ (see Proposition~\ref{rmk_eta}) and $\eta(x,t)=n(x,t)+p(x,t)\geq 0$ for $x\in (-1,1), t>0$, which implies $\int_{-1}^{1}{\eta {{{\bar{D}}}^{2}}dx\ge 0}$ a crucial inequality for the use of (\ref{id74}) to prove Theorem~\ref{nl-stability-thm}. In physical point of view, the nonnegativeness of $n_0$ and $p_0$ is reasonable because $n_0$ and $p_0$ are concentrations of anions and cations, respectively, at the initial time $t=0$.

The rest of this paper is organized as follows: For linear stability, we prove Theorem~\ref{prop1} and Corollary~\ref{cor1} in Section~\ref{lzp}. In Section~\ref{nl-stab}, the proofs of Theorem~\ref{nl-stability-thm} and Corollary~\ref{cor4-1} are provided for nonlinear stability.

\section{Proof of linear stability}\label{lzp}
\ \ \ \ In this section, we study (\ref{id15}) with (\ref{id16}), which is the linearized problem of (\ref{id12}) with (\ref{id13}). We derive the energy law of $(D,H)$ as follows:
\begin{thm}\label{thm1}
Let $(\tilde{\delta},\tilde{\eta},\tilde{\phi})$ be the solution of
(\ref{id15}) with boundary conditions (\ref{id16}). If (\ref{id20}) holds true,
then $(D,H)$ satisfies
\begin{align}
\frac{1}{2} \frac{d}{dt} \int_{-1}^{1} (D^2 + H^2) dx &= -\int_{-1}^1 \Big(D_x^2 + H_x^2 + \frac{1}{\varepsilon} \eta^0 D^2\Big) dx - \tilde{\phi}_x (-1,t) \int_{-1}^1 (\eta^0 D + \delta^0 H) dx \label{id23}
\end{align}
for $t>0$.
\end{thm}

\noindent\begin{proof}
For equation (\ref{id24}), we multiply it by $D$, integrate it from $-1$ to $1$, and use
integration by parts. Then
\begin{align}\label{id26}
\frac{1}{2} \frac{d}{dt} \int_{-1}^{1} D^2 dx = -\int_{-1}^1 \bigg(D_x^2 + \eta^0 \tilde{\phi}_x D + \psi_x H_x D\bigg) dx.
\end{align}
Here we have used the fact that $D=0$ at $x=\pm 1$ from (\ref{0-DBC}). On the other hand, we integrate $\varepsilon \tilde{\phi}_{xx} = \tilde{\delta}$ the Poisson
equation of (\ref{id15}) from $-1$ to $x$. Then
\begin{align*}
\varepsilon (\tilde{\phi}_x (x,t) - \tilde{\phi}_x (-1,t)) &= \varepsilon \int_{-1}^x \tilde{\phi}_{yy} (y,t) dy \\
&= \int_{-1}^x \tilde{\delta} (y,t) dy \\
&= D(x,t),
\end{align*}
which gives
\begin{align}\label{id27}
\tilde{\phi}_x (x,t) = \frac{1}{\varepsilon} D(x,t) + \tilde{\phi}_x (-1,t)
\end{align}
for $x \in (-1,1)$, $t>0$. Consequently,
\begin{align}\label{id28}
\int_{-1}^1 \eta^0 \tilde{\phi}_x D dx = \frac{1}{\varepsilon} \int_{-1}^1 \eta^0 D^2 dx + \tilde{\phi}_x (-1,t) \int_{-1}^1 \eta^0 D dx.
\end{align}

For equation (\ref{id25}), we multiply it by $H$, integrate it
from $-1$ to $1$, and use integration by parts. Then
\begin{align}\label{id29}
\frac{1}{2} \frac{d}{dt} \int_{-1}^{1} H^2 dx = -\int_{-1}^1 \bigg(H_x^2 + \delta^0 \tilde{\phi}_x H + \psi_x D_x H \bigg) dx.
\end{align}
Here we have used the fact that $H=0$ at $x=\pm 1$ from (\ref{0-DBC}). Moreover, we use (\ref{pd1}), (\ref {id27}) and integration by parts to get
\begin{align}
\quad \int_{-1}^1 (\delta^0 \tilde{\phi}_x H + \psi_x D_x H) dx &= \int_{-1}^1 \Big(\varepsilon \psi_{xx} \tilde{\phi}_x + \varepsilon\psi_x \tilde{\phi}_{xx}\Big)H dx \notag\\
&= \int_{-1}^1 \varepsilon (\psi_x \tilde{\phi}_x)_x H dx \notag\\
&= -\int_{-1}^1 \varepsilon \psi_x \tilde{\phi}_x H_x dx \notag \\
&= -\int_{-1}^1 \psi_x D H_x dx - \varepsilon\tilde{\phi}_x (-1,t) \int_{-1}^1 \psi_x H_x dx \notag\\
&= -\int_{-1}^1 \psi_x D H_x dx + \tilde{\phi}_x (-1,t) \int_{-1}^1 \delta^0 H dx\,.  \label{id30}
\end{align}
Therefore, the proof of (\ref{id23}), i.e., Theorem~\ref{thm1} is complete by combining (\ref{id26}), (\ref{id28}), (\ref{id29}) and (\ref{id30}).
\end{proof}

In order to use Theorem~\ref{thm1} for the proof of the linear stability of $(n^0,p^0,\psi) $, we need to consider $\tilde{\phi}_x (-1,t)$ the gradient estimate at the boundary point $x=-1$. Notice that $\tilde{\phi}$ satisfies
\begin{align}\label{id31}
\varepsilon \tilde{\phi}_{xx} = \tilde{\delta}\quad\hbox{ for }\quad x\in (-1,1)\,,
\end{align}
with Robin boundary condition
\begin{align}\label{id32}
\tilde{\phi} \pm \gamma_\varepsilon \tilde{\phi}_x = 0 \quad \mbox{at} \quad x = \pm 1,
\end{align}
for each $t > 0$. Fix $t>0$ arbitrarily. Then we integrate both sides of equation (\ref{id31}) in $x$ over the interval $(-1,1)$, and get
\begin{align*}
\varepsilon (\tilde{\phi}_x (1,t) - \tilde{\phi}_x (-1,t)) = \varepsilon \int_{-1}^1 \tilde{\phi}_{xx} dx =  \int_{-1}^1 \tilde{\delta} dx = D(1,t) = 0\,,
\end{align*}
which implies
\begin{align}\label{id33}
\tilde{\phi}_x (1,t) = \tilde{\phi}_x (-1,t)\,.
\end{align}
Here we have used (\ref{id21}) and (\ref{0-DBC}). Thus (\ref{id32}) and (\ref{id33}) give
\begin{align}\label{id34}
\tilde{\phi} (1,t) = -\gamma_\varepsilon \tilde{\phi}_x (1,t) = -\gamma_\varepsilon \tilde{\phi}_x (-1,t) = -\tilde{\phi} (-1,t).
\end{align}
By (\ref{id33}) and (\ref{id34}), we have
\begin{align}
2\tilde{\phi}_x (-1,t) &= \int_{-1}^1 (x\tilde{\phi}_x)_x dx \notag\\
&= \int_{-1}^1 (\tilde{\phi}_x + x\tilde{\phi}_{xx}) dx \notag\\
&= -2\tilde{\phi} (-1,t) + \int_{-1}^1 x\tilde{\phi}_{xx} dx \notag\\
&= -2\gamma_\varepsilon \tilde{\phi}_x (-1,t) + \int_{-1}^1 x\tilde{\phi}_{xx} dx\,. \label{id35}
\end{align}
Furthermore, we use (\ref{id21}) (which gives ${{D}_{x}}=\tilde{\delta }$), (\ref{id31}) and integration by parts to get
\begin{align}
\int_{-1}^1 x\tilde{\phi}_{xx} dx &= \frac{1}{\varepsilon} \int_{-1}^1 x\tilde{\delta} dx \notag\\
&= \frac{1}{\varepsilon} \int_{-1}^1 x D_x dx \notag \\
&= -\frac{1}{\varepsilon} \int_{-1}^1 D dx\,. \label{id36}
\end{align}
Consequently, (\ref{id35}) and (\ref{id36}) imply
\begin{align}\label{id37}
\tilde{\phi}_x (-1,t) = -\frac{1}{2(1+\gamma_\varepsilon)\varepsilon} \int_{-1}^1 D dx\,.
\end{align}
Note that ${{\gamma }_{\varepsilon }}\sim \sqrt{\varepsilon }\gamma $ as $\varepsilon \to 0+$ (see Theorem~A). Thus
as $\varepsilon$ approaches zero, $\left|\tilde{\phi}_x (-1,t)\right|$ becomes extremely large if the integral $\int_{-1}^1 D dx$ is away from zero. This makes (\ref{id23}) hard to be used for the proof of the linear stability of $(n^0,p^0,\psi)$.

To overcome the difficulty, we transform $D$ and $H$ into $\bar{D}$ and $\bar{H}$ by truncating the average of $D$ and $H$, respectively:
\begin{align}
\bar{D} = D - \frac{1}{2} \int_{-1}^1 D dx, \label{id38}\\
\bar{H} = H - \frac{1}{2} \int_{-1}^1 H dx, \label{id39}
\end{align}
and
\begin{align}
d = \frac{1}{2} \int_{-1}^1 D dx\,, \label{id40}\\
h = \frac{1}{2} \int_{-1}^1 H dx\,, \label{id41}
\end{align}
where $d$ and $h$ are the average of $D$ and $H$ at time $t$, respectively. Note that $\bar{D} = D - d$ and $\bar{H} = H - h$ satisfy
\BE\label{av-DH1}
\int_{-1}^{1}{\bar{D}dx=\int_{-1}^{1}{\bar{H}dx=0}}\,,
\EE
${{\bar{D}}_{x}}={{D}_{x}}=\tilde{\delta }$ and ${{\bar{H}}_{x}}={{H}_{x}}=\tilde{\eta }$ for $x\in (-1,1)$ and $t>0$.

Now we state the energy law for $(\bar{D},\bar{H},d,h)$ as follows:
\begin{thm}\label{thm2}
Under the same assumption as Theorem \ref{thm1}, we have
\begin{align}
&\quad\ \frac{d}{dt} \bigg[ \frac{1}{2}\int_{-1}^1 (\bar{D}^2 + \bar{H}^2) dx + d^2 + h^2 \bigg] \label{id42}\\
&= -\int_{-1}^1 \Big(\bar{D}_x^2 + \bar{H}_x^2 + \frac{1}{\varepsilon} \eta^0 \bar{D}^2 \Big)dx - \frac{2m_0 \gamma_\varepsilon d^2}{(1+\gamma_\varepsilon)\varepsilon} \notag\\
&\quad\quad- \frac{(1+2\gamma_\varepsilon)d}{(1+\gamma_\varepsilon)\varepsilon} \int_{-1}^1 \eta^0 \bar{D} dx + \frac{d}{(1+\gamma_\varepsilon)\varepsilon} \int_{-1}^1 \delta^0 \bar{H} dx \notag
\end{align}
for $t>0$.
\end{thm}
\begin{proof}
Note that
\begin{align}
\int_{-1}^1 \bar{D}^2 dx &= \int_{-1}^1 (D-d)^2 dx = \int_{-1}^1 D^2 dx - 2d^2, \label{id43}\\
\int_{-1}^1 \bar{H}^2 dx &= \int_{-1}^1 (H-h)^2 dx = \int_{-1}^1 H^2 dx - 2h^2. \label{id44}
\end{align}
By (\ref{id43}) and (\ref{id44}), equation (\ref{id23}) becomes
\begin{align}
&\quad\ \frac{d}{dt} \bigg[ \frac{1}{2}\int_{-1}^1 (\bar{D}^2 + \bar{H}^2) dx + d^2 + h^2 \bigg] \notag\\
&= \frac{1}{2} \frac{d}{dt} \int_{-1}^1 (D^2 + H^2) dx \label{id45}\\
&= -\int_{-1}^1 \Big(D_x^2 + H_x^2 + \frac{1}{\varepsilon} \eta^0 D^2\Big) dx - \tilde{\phi}_x (-1,t) \int_{-1}^1 (\eta^0 D + \delta^0 H) dx. \notag
\end{align}
Then we put (\ref{id37}) into (\ref{id45}) and get
\begin{align}
&\quad\ \frac{d}{dt} \bigg[ \frac{1}{2}\int_{-1}^1 (\bar{D}^2 + \bar{H}^2) dx + d^2 + h^2 \bigg] \notag\\
&= -\int_{-1}^1 \Big(\bar{D}_x^2 + \bar{H}_x^2 + \frac{1}{\varepsilon} \eta^0 (\bar{D}+d)^2\Big) dx + \frac{d}{(1+\gamma_\varepsilon)\varepsilon} \int_{-1}^1 (\eta^0 (\bar{D}+d) + \delta^0 (\bar{H}+h)) dx \notag\\
&= -\int_{-1}^1 \Big(\bar{D}_x^2 + \bar{H}_x^2 + \frac{1}{\varepsilon} \eta^0 \bar{D}^2 \Big)dx - \frac{1}{\varepsilon} \int_{-1}^1 \eta^0 (2d\bar{D} + d^2) dx \notag \\ 
&\quad\quad+ \frac{d}{(1+\gamma_\varepsilon)\varepsilon} \int_{-1}^1 (\eta^0 \bar{D} + \delta^0 \bar{H})dx + \frac{d^2}{(1+\gamma_\varepsilon)\varepsilon} \int_{-1}^1 \eta^0 dx + \frac{dh}{(1+\gamma_\varepsilon)\varepsilon} \int_{-1}^1 \delta^0 dx \notag\\
&= -\int_{-1}^1 \Big(\bar{D}_x^2 + \bar{H}_x^2 + \frac{1}{\varepsilon} \eta^0 \bar{D}^2 \Big)dx - \frac{2m_0 \gamma_\varepsilon d^2}{(1+\gamma_\varepsilon)\varepsilon} \notag\\
&\quad\quad- \frac{(1+2\gamma_\varepsilon)d}{(1+\gamma_\varepsilon)\varepsilon} \int_{-1}^1 \eta^0 \bar{D} dx + \frac{d}{(1+\gamma_\varepsilon)\varepsilon} \int_{-1}^1 \delta^0 \bar{H} dx\,, \notag
\end{align}
which gives (\ref{id42}) and complete the proof of Theorem~\ref{thm2}. Here the last equality uses the fact that $d$ and $h$ are independent of $x$, and
\begin{align*}
\int_{-1}^1 \eta^0 dx = 2m_0,\quad \int_{-1}^1 \delta^0 dx = 0\,.
\end{align*}
\end{proof}

\subsection{Proof of Theorem~\ref{prop1}}
\ \ \ \
In order to use Theorem~\ref{thm2} for the proof of stability, we need to estimate the integral terms of (\ref{id42}) involving ${{\eta }^{0}}$ and ${{\delta }^{0}}$, where ${{\eta }^{0}}=\frac{{{m}_{0}}{{e}^{\psi }}}{\int_{-1}^{1}{{{e}^{\psi }}}}+\frac{{{m}_{0}}{{e}^{-\psi }}}{\int_{-1}^{1}{{{e}^{-\psi }}}}$, ${{\delta }^{0}}=\frac{{{m}_{0}}{{e}^{\psi }}}{\int_{-1}^{1}{{{e}^{\psi }}}}-\frac{{{m}_{0}}{{e}^{-\psi }}}{\int_{-1}^{1}{{{e}^{-\psi }}}}$ and $\psi$ is the solution of (\ref{id3-p}) satisfying
\BE\label{id3-p-1}
\varepsilon {{\psi }_{xx}}={{\delta }^{0}} \quad\hbox{ for }\quad x\in (-1,1)\,.
\EE
By Theorem A,
\begin{align}\label{id48}
|\psi(x)| \leq \phi_0(1) \left(e^{-\frac{M}{\sqrt{\varepsilon}}(1+x)} + e^{-\frac{M}{\sqrt{\varepsilon}}(1-x)} \right)
\end{align}
for all $x \in [-1,1]$ and $\varepsilon > 0$, where $M=\sqrt{\frac{{{m}_{0}}}{2{{e}^{{{\phi }_{0}}\left( 1 \right)}}}}$ independent of $\varepsilon$ and $\gamma$. Then we claim that
\begin{align}\label{id48-1}
\int_{-1}^1 e^{\pm \psi(x)} dx \leq 2 + K_0\phi_0(1)\frac{\sqrt{\varepsilon}}{M},
\end{align}
where $K_0$ is a positive constant defined as follows:
\[{{K}_{0}}=\underset{0<\left| y \right|\le {{\phi }_{0}}\left( 1 \right)}{\mathop{\sup }}\,\frac{\left| {{e}^{y}}-1 \right|}{\left| y \right|}>0\,. \]
Note that $K_0$ only depends on $\phi_0(1)$. By Theorem A, $\psi$ is odd, increasing, and $\left|\psi(x)\right|\leq\phi_0(1)$ for $x\in [-1,1]$, which implies
\begin{align}\label{int-exp-psi}
\int_{-1}^1 e^{\psi(x)} dx = \int_{-1}^1 e^{-\psi(x)} dx
\end{align}
and
\begin{align}\label{id48-2}
e^{\psi(x)} \leq \left\{
\begin{array}{ll}
1 + K_0 \psi(x), & x \geq 0, \\
1, & x < 0.
\end{array}
\right.
\end{align}
And then we may use (\ref{id48}) to get (\ref{id48-1}). Now we claim that
\begin{align}\label{et0-m0-1}
\left|\eta^0 - m_0 - |\eta^0 - m_0|\right| \leq m_0 K_0\phi_0(1)\frac{\sqrt{\varepsilon}}{M}.
\end{align}
We divide the domain interval $[-1,1]$ into two parts as follows:
\begin{align*}
A = \left\{x\in[-1,1]:\eta^0 \geq m_0 \right\}
\end{align*}
and
\begin{align*}
B = \left\{x\in[-1,1]:\eta^0 < m_0 \right\}.
\end{align*}
Then we get
\begin{align*}
\left|\eta^0 - m_0 - |\eta^0 - m_0|\right| = 0 \mbox{ on } A,
\end{align*}
and
\begin{align*}
\left|\eta^0 - m_0 - |\eta^0 - m_0|\right| = 2\left(m_0-\eta^0\right) \mbox{ on } B.
\end{align*}
On $B$, by the definition of $\eta^0$ and (\ref{int-exp-psi}), we get
\begin{align*}
m_0-\eta^0 &= m_0 \left[1 -\frac{1}{\int_{-1}^1 e^\psi}\left(e^{\psi} + e^{-\psi}\right) \right] \\
&\leq m_0 \left[1 -\frac{2}{\int_{-1}^1 e^\psi} \right] \\
&\leq m_0 \left[1 -\frac{1}{1 + K_0\phi_0(1)\frac{\sqrt{\varepsilon}}{2M}} \right] \\
&\leq m_0 K_0\phi_0(1)\frac{\sqrt{\varepsilon}}{2M}.
\end{align*}
Here we have used the fact that $e^\psi + e^{-\psi} \geq 2$ and (\ref{id48-1}). Therefore, we complete the proof of (\ref{et0-m0-1}). Moreover, by (\ref{et0-m0-1}) and the fact that $\int_{-1}^1 \eta^0 = 2m_0$, we have
\begin{align}\label{et0-m0-2}
\int_{-1}^1 \left|\eta^0 - m_0\right|dx \leq 2m_0 K_0\phi_0(1)\frac{\sqrt{\varepsilon}}{M}.
\end{align}

To get the gradient estimate of $\psi$, we multiply (\ref{id3-p-1}) by $\psi_x$ and integrate it over $(-1,x)$. Then
\begin{align}\label{id53}
\frac{\varepsilon}{2} \psi_x^2(x) = \frac{m_0}{\int_{-1}^1 e^\psi} \left(e^\psi + e^{-\psi}\right) + C_\varepsilon\,,\quad\hbox{ for }\quad x\in (-1,1)\,,
\end{align}
where $C_\varepsilon$ is a constant depending on $\varepsilon$. Taking the value at $x=0$ for both sides of (\ref{id53}), we have
\begin{align}
C_\varepsilon = \frac{\varepsilon}{2} \psi_x^2(0) - \frac{2m_0}{\int_{-1}^1 e^\psi}. \label{C-epsilon}
\end{align}
Integrate both sides of (\ref{id53}) and by (\ref{C-epsilon}), we get
\begin{align*}
\frac{\varepsilon}{2} \int_{-1}^1 \psi_x^2 = 2m_0 \left(1-\frac{2}{\int_{-1}^1 e^\psi}\right) + \varepsilon\psi_x^2(0).
\end{align*}
Then using (\ref{id48-1}), we have
\begin{align}
\frac{\varepsilon}{2} \int_{-1}^1 \psi_x^2 &\leq 2m_0 \left(1-\frac{1}{1 + K_0\phi_0(1)\frac{\sqrt{\varepsilon}}{2M}}\right) + \varepsilon\psi_x^2(0) \notag\\
&\leq m_0 K_0\phi_0(1)\frac{\sqrt{\varepsilon}}{M} + \varepsilon\psi_x^2(0). \label{est-grad-psi-0}
\end{align}
By Theorem A, $\psi$ is increasing on $[-1,1]$, convex in $(0,1)$, and concave in $(-1,0)$, which implies $\psi_x \geq \psi_x(0) \geq 0$. By the mean value theorem and (\ref{id48}),
\begin{align}
\psi_x(0) &\leq \frac{\psi\left(\frac{1}{2}\right) - \psi(0)}{\frac{1}{2}} \notag\\
&= 2\psi\left(\frac{1}{2}\right) \notag\\
&\leq 2\phi_0(1) \left(e^{-\frac{3M}{2\sqrt{\varepsilon}}} + e^{-\frac{M}{2\sqrt{\varepsilon}}} \right). \label{est-psi-x-0}
\end{align}
Therefore, (\ref{est-grad-psi-0}) implies
\begin{align}
\int_{-1}^1 \psi_x^2 \leq 2m_0 K_0\phi_0(1)\frac{1}{M\sqrt{\varepsilon}} + 4\phi_0^2(1) \left(e^{-\frac{3M}{2\sqrt{\varepsilon}}} + e^{-\frac{M}{2\sqrt{\varepsilon}}} \right)^2. \label{est-grad-psi-1}
\end{align}
And we have
\begin{align}
\int_{-1}^1 \psi_x^2 \leq 3m_0 K_0\phi_0(1)\frac{1}{M\sqrt{\varepsilon}} \label{est-grad-psi-2}
\end{align}
for $0 < \varepsilon < \varepsilon_0$, where $\varepsilon_0$ is a positive constant depending only on $m_0$ and $\phi_0(1)$.

Moreover, we use (\ref{id3-p-1}) and integration by parts to get
\begin{align*}
\int_{-1}^1 \delta^0 \bar{H} dx = \varepsilon \int_{-1}^1 \psi_{xx} \bar{H} dx = -\varepsilon \int_{-1}^1 \psi_x \bar{H}_x dx.
\end{align*}
The boundary integral of the last equality is zero because $\psi_x$ is even and $\bar{H}(-1) = \bar{H}(1) = -h$. Hence by (\ref{est-grad-psi-2}) and H\"{o}lder's inequality, we have
\begin{eqnarray}
\bigg|\frac{d}{(1+{{\gamma }_{\varepsilon }})\varepsilon }\int_{-1}^{1}{{{\delta }^{0}}}\bar{H}dx\bigg| &\le& \bigg|\frac{d}{1+{{\gamma }_{\varepsilon }}}\int_{-1}^{1}{{{\psi }_{x}}}{{{\bar{H}}}_{x}}dx\bigg| \notag \\
&\le& \frac{d}{1+{{\gamma }_{\varepsilon }}}{{\left(\int_{-1}^{1}{\psi _{x}^{2}}dx\right)}^{1/2}}{{\left(\int_{-1}^{1}{\bar{H}_{x}^{2}}dx\right)}^{1/2}} \notag \\
&\le& \frac{1}{2}\int_{-1}^{1}{\bar{H}_{x}^{2}}dx+\frac{{{d}^{2}}}{2(1+{{\gamma }_{\varepsilon }})}\int_{-1}^{1}{\psi _{x}^{2}}dx \notag \\
&\le& \frac{1}{2}\int_{-1}^{1}{\bar{H}_{x}^{2}}dx+\frac{3{{d}^{2}}}{2(1+{{\gamma }_{\varepsilon }})}{{m}_{0}}{{K}_{0}}{{\phi }_{0}}\left( 1 \right) \frac{1}{M\sqrt{\varepsilon}}\,. \label{id52}
\end{eqnarray}
Here we have used the fact that $\gamma_\varepsilon>0$.

Now we want to estimate the integral $\bigg|\frac{(1+2\gamma_\varepsilon)d}{(1+\gamma_\varepsilon)\varepsilon} \int_{-1}^1 \eta^0 \bar{D} dx\bigg|$ as follows:

\noindent Due to $\int_{-1}^1 \bar{D} dx = 0$,
\begin{align*}
\int_{-1}^1 \eta^0 \bar{D} dx = \int_{-1}^1 (\eta^0 - m_0) \bar{D} dx\,,
\end{align*}
so we may use the H\"{o}lder's inequality and $ab\le {{a}^{2}}+\frac{1}{4}{{b}^{2}}$, $\forall a,b\geq 0$, to get
\begin{align}
&\quad\ \bigg|\frac{(1+2\gamma_\varepsilon)d}{(1+\gamma_\varepsilon)\varepsilon} \int_{-1}^1 \eta^0 \bar{D} dx\bigg| \notag\\
&= \bigg|\frac{(1+2\gamma_\varepsilon)d}{(1+\gamma_\varepsilon)\varepsilon} \int_{-1}^1 (\eta^0 - m_0) \bar{D} dx\bigg| \notag \\
&\leq \bigg|\frac{(1+2\gamma_\varepsilon)d}{(1+\gamma_\varepsilon)\varepsilon}\bigg| \bigg(\int_{-1}^1 |\eta^0 - m_0|dx\bigg)^{1/2} \bigg(\int_{-1}^1 |\eta^0 - m_0| \bar{D}^2 dx\bigg)^{1/2} \notag\\
&\leq \frac{1}{\varepsilon} \int_{-1}^1 |\eta^0 - m_0| \bar{D}^2 dx + \frac{(1+2\gamma_\varepsilon)^2 d^2}{4(1+\gamma_\varepsilon)^2 \varepsilon} \int_{-1}^1 |\eta^0 - m_0|dx\,. \label{id49}
\end{align}
For the integral $\frac{{{(1+2{{\gamma }_{\varepsilon }})}^{2}}{{d}^{2}}}{4{{(1+{{\gamma }_{\varepsilon }})}^{2}}\varepsilon }\int_{-1}^{1}{|}{{\eta }^{0}}-{{m}_{0}}|dx$, we use (\ref{et0-m0-2})
\begin{align*}
\int_{-1}^1 \left|\eta^0 - m_0\right|dx \leq 2m_0 K_0\phi_0(1)\frac{\sqrt{\varepsilon}}{M},
\end{align*}
to get
\begin{align*}
\frac{(1+2\gamma_\varepsilon)^2 d^2}{4(1+\gamma_\varepsilon)^2\varepsilon} \int_{-1}^1 \left|\eta^0 - m_0\right|dx \leq \frac{(1+2\gamma_\varepsilon)^2 d^2}{2(1+\gamma_\varepsilon)^2 \sqrt{\varepsilon}}\frac{m_0 K_0\phi_0(1)}{M}.
\end{align*}
Hence (\ref{id49}) becomes
\BE\label{int-etm0}
\bigg|\frac{(1+2\gamma_\varepsilon)d}{(1+\gamma_\varepsilon)\varepsilon} \int_{-1}^1 \eta^0 \bar{D} dx\bigg|
\leq \frac{1}{\varepsilon} \int_{-1}^1 |\eta^0 - m_0| \bar{D}^2 dx +
\frac{(1+2\gamma_\varepsilon)^2 d^2}{2(1+\gamma_\varepsilon)^2 \sqrt{\varepsilon}}\frac{m_0 K_0\phi_0(1)}{M}.
\EE
Furthermore, we may use (\ref{et0-m0-1})
\begin{align*}
\left|\eta^0 - m_0 - |\eta^0 - m_0|\right| \leq m_0 K_0\phi_0(1)\frac{\sqrt{\varepsilon}}{M}
\end{align*}
and (\ref{int-etm0}) to get
\begin{eqnarray}
  & & -\frac{1}{\varepsilon }\int_{-1}^{1}{{{\eta }^{0}}{{{\bar{D}}}^{2}}}-\frac{(1+2{{\gamma }_{\varepsilon }})d}{(1+{{\gamma }_{\varepsilon }})\varepsilon }\int_{-1}^{1}{{{\eta }^{0}}}\bar{D}dx\le -\frac{1}{\varepsilon }\int_{-1}^{1}{{{\eta }^{0}}{{{\bar{D}}}^{2}}}+\left|\frac{(1+2{{\gamma }_{\varepsilon }})d}{(1+{{\gamma }_{\varepsilon }})\varepsilon }\int_{-1}^{1}{{{\eta }^{0}}}\bar{D}dx\right| \notag \\
 & &\le -\frac{1}{\varepsilon }\int_{-1}^{1}{{{\eta }^{0}}{{{\bar{D}}}^{2}}}+\frac{1}{\varepsilon }\int_{-1}^{1}{|}{{\eta }^{0}}-{{m}_{0}}|{{{\bar{D}}}^{2}}dx+\frac{(1+2\gamma_\varepsilon)^2 d^2}{2(1+\gamma_\varepsilon)^2 \sqrt{\varepsilon}}\frac{m_0 K_0\phi_0(1)}{M} \notag \\
 & & =-\frac{1}{\varepsilon }\int_{-1}^{1}{\left[ {{m}_{0}}+\left( {{\eta }^{0}}-{{m}_{0}}-\left| {{\eta }^{0}}-{{m}_{0}} \right| \right) \right]}{{{\bar{D}}}^{2}}+\frac{(1+2\gamma_\varepsilon)^2 d^2}{2(1+\gamma_\varepsilon)^2 \sqrt{\varepsilon}}\frac{m_0 K_0\phi_0(1)}{M} \notag \\
 & & \le -\frac{1}{\varepsilon }{{m}_{0}}\left[ 1- K_0\phi_0(1)\frac{\sqrt{\varepsilon}}{M} \right]\int_{-1}^{1}{{{{\bar{D}}}^{2}}}+\frac{(1+2\gamma_\varepsilon)^2 d^2}{2(1+\gamma_\varepsilon) \sqrt{\varepsilon}}\frac{m_0 K_0\phi_0(1)}{M}\,. \label{crustp1}
\end{eqnarray}
Again, here we have used $\gamma_\varepsilon > 0$. And we can choose $\varepsilon_1 > 0$ depending only on $m_0$ and $\phi_0(1)$ such that
\begin{align*}
1 - K_0\phi_0(1)\frac{\sqrt{\varepsilon}}{M} > \frac{1}{2}
\end{align*}
for $0 < \varepsilon < \varepsilon_1$. We substitute (\ref{id52}) and (\ref{crustp1}) into (\ref{id42}). Then
\begin{eqnarray}\label{rv-eng1}
\frac{d}{dt}\left[\frac{1}{2}\int_{-1}^{1}{({{{\bar{D}}}^{2}}+{{{\bar{H}}}^{2}})}dx+{{d}^{2}}+{{h}^{2}}\right] &\le& -\int_{-1}^{1}{\left( \bar{D}_{x}^{2}+\frac{1}{2}\bar{H}_{x}^{2} \right)} \\
& &~~-\frac{1}{\varepsilon }{{m}_{0}}\left[ 1- K_0\phi_0(1)\frac{\sqrt{\varepsilon}}{M} \right]\int_{-1}^{1}{{{{\bar{D}}}^{2}}}  \notag \\
& &~~+\frac{(1+2\gamma_\varepsilon)^2 d^2}{2(1+\gamma_\varepsilon) \sqrt{\varepsilon}}\frac{m_0 K_0\phi_0(1)}{M} \notag \\
& &~~-\frac{2{{m}_{0}}{{\gamma }_{\varepsilon }}{{d}^{2}}}{(1+{{\gamma }_{\varepsilon }})\varepsilon }+\frac{3{{d}^{2}}}{2(1+{{\gamma }_{\varepsilon }})}{{m}_{0}}{{K}_{0}}{{\phi }_{0}}\left( 1 \right) \frac{1}{M\sqrt{\varepsilon}}\,. \notag
\end{eqnarray}
Recall that
\begin{enumerate}
\item[(\ref{id52}):]
$$\bigg| \frac{d}{(1+\gamma_\varepsilon)\varepsilon} \int_{-1}^1 \delta^0 \bar{H} dx \bigg| \leq \frac{1}{2} \int_{-1}^1 \bar{H}_x^2 dx + \frac{3{{d}^{2}}}{2(1+{{\gamma }_{\varepsilon }})}{{m}_{0}}{{K}_{0}}{{\phi }_{0}}\left( 1 \right) \frac{1}{M\sqrt{\varepsilon}},$$
\item[(\ref{id42}):]
$$\begin{array}{rll}
\frac{d}{dt}\left[\frac{1}{2}\int_{-1}^{1}{({{{\bar{D}}}^{2}}+{{{\bar{H}}}^{2}})}dx+{{d}^{2}}+{{h}^{2}}\right] = &-\int_{-1}^{1}{(}\bar{D}_{x}^{2}+\bar{H}_{x}^{2}+\frac{1}{\varepsilon }{{\eta }^{0}}{{{\bar{D}}}^{2}})dx-\frac{2{{m}_{0}}{{\gamma }_{\varepsilon }}{{d}^{2}}}{(1+{{\gamma }_{\varepsilon }})\varepsilon } \\ & \\
 & -\frac{(1+2{{\gamma }_{\varepsilon }})d}{(1+{{\gamma }_{\varepsilon }})\varepsilon }\int_{-1}^{1}{{{\eta }^{0}}}\bar{D}dx+\frac{d}{(1+{{\gamma }_{\varepsilon }})\varepsilon }\int_{-1}^{1}{{{\delta }^{0}}}\bar{H}dx.
\end{array}$$
\end{enumerate}
In order to get (\ref{id47}), we need the nonpositiveness of the last three terms of (\ref{rv-eng1}) together, which may hold true by assuming $\frac{\gamma_\varepsilon}{\sqrt{\varepsilon}} > \frac{\left((1+2\gamma_{max})^2+3\right)K_0\phi_0(1)}{4M}$. Here we have used the fact that $\gamma_\varepsilon \leq \gamma_{max}$. Therefore, we set $\tilde{\varepsilon }=\min \left\{\varepsilon_0,\varepsilon_1\right\}$ ($\varepsilon_0,\varepsilon_1>0$ only depends on $m_0$ and $\phi_0(1)$) and complete the proof of Theorem~\ref{prop1}.

\subsection{Proof of Corollary~\ref{cor1} }
\medskip

Recall that since $D(\pm 1,t) = H(\pm 1,t) = 0$ and $\int_{-1}^1 \bar{D} dx = \int_{-1}^1 \bar{H} dx = 0$, we may use the Poincar\'{e}'s inequality to get
\begin{align*}
&\|D\|_{L_x^2} \le C \|D_x\|_{L_x^2},\quad \|\bar{D}\|_{L_x^2} \le C \|\bar{D}_x\|_{L_x^2},\\
&\|H\|_{L_x^2} \le C \|H_x\|_{L_x^2},\quad \|\bar{H}\|_{L_x^2} \le C \|\bar{H}_x\|_{L_x^2},
\end{align*}
where $C$ is a positive constant from the Poincar\'{e}'s inequality. Hence by the H\"{o}lder inequality, we have
$$
\left| d \right|=\left| \frac{1}{2}\int_{-1}^{1}{D} \right|\le \frac{1}{2}{{\left\| D \right\|}_{L_{x}^{2}}}{{\left( \int_{-1}^{1}{1} \right)}^{1/2}}=\frac{\sqrt{2}}{2}{{\left\| D \right\|}_{L_{x}^{2}}},
$$
which implies $|d|\leq \frac{\sqrt{2}C}{2}{{\| {{{\bar{D}}}_{x}} \|}_{L_{x}^{2}}}$ since $D_x = \bar{D}_x$. Similarly, $|h|\leq \frac{\sqrt{2}C}{2}{{\| {{{\bar{H}}}_{x}} \|}_{L_{x}^{2}}}$. Hence (\ref{id47}) implies
\begin{align}\notag
\frac{d}{dt} \bigg[ \frac{1}{2}\int_{-1}^1 (\bar{D}^2 + \bar{H}^2) dx + d^2 + h^2 \bigg] \leq -\alpha \bigg[ \frac{1}{2}\int_{-1}^1 (\bar{D}^2 + \bar{H}^2) dx + d^2 + h^2 \bigg]
\end{align}
for some $\alpha > 0$ depending only on the constants from the Poincar\'{e}'s inequality. Therefore, we obtain (\ref{id55}) and complete the proof.

\section{Proof of nonlinear stability}\label{nl-stab}
\ \ \ \ To get nonlinear stability, we generalize the idea of linear stability to study $(\tilde{\delta},\tilde{\eta},\tilde{\phi})$ the solution of nonlinear system (\ref{id12}) with boundary condition (\ref{id13}), which also has conservation laws as follows:
\begin{align}\label{id64}
\frac{d}{dt} \int_{-1}^{1} \tilde{\delta} dx = \frac{d}{dt} \int_{-1}^{1} \tilde{\eta} dx = 0 \quad \mbox{for} \quad t > 0\,.
\end{align}
As for linear stability, we assume the initial data of $(\tilde{\delta},\tilde{\eta})$ satisfying
\begin{align}
\int_{-1}^{1} \tilde{\delta} (x,0) dx = \int_{-1}^{1} \tilde{\eta} (x,0) dx = 0\,, \label{id65}
\end{align}
which is same as (\ref{id20}). Moreover, as for Theorem~\ref{prop1}, we set
\begin{align*}
D(x,t) = \int_{-1}^{x} \tilde{\delta}(s,t) ds,\quad H(x,t) = \int_{-1}^{x} \tilde{\eta}(s,t) ds \quad\hbox{ for }\quad x\in (-1,1)\,,\: t>0\,,
\end{align*}
$\bar{D}=D-d$ and $\bar{H} = H-h$, where $d = \frac{1}{2} \int_{-1}^1 D dx$ and $h = \frac{1}{2} \int_{-1}^1 H dx$.

To control the nonlinear terms $\tilde{\eta }{{\tilde{\phi }}_{x}}$ and $\tilde{\delta }{{\tilde{\phi }}_{x}}$ of system~(\ref{id12}), we assume that the initial data satisfies $n\left( x,0 \right)={{n}^{0}}\left( x \right)+\tilde{n}\left( x,0 \right)$, $p\left( x,0 \right)={{p}^{0}}\left( x \right)+\tilde{p}\left( x,0 \right)\ge 0$ for $x\in (-1,1)$ and (\ref{nonlinear-I0}), which implies that the right side of (\ref{nl-stability}) becomes negative (see Theorem~\ref{nl-stability-thm}). To prove Theorem~\ref{nl-stability-thm}, we first derive energy laws as for linear stability in Section~\ref{lzp}. Such energy laws are represented as follows:
\begin{thm}\label{thm5}
If $(\tilde{\delta},\tilde{\eta},\tilde{\phi})$ is a solution of (\ref{id12}) with boundary condition (\ref{id13}), and the initial data satisfy (\ref{id65}). Then we have
\begin{align}
&\quad\ \frac{1}{2} \frac{d}{dt} \int_{-1}^{1} (D^2 + H^2) dx \label{id73}\\
&= -\int_{-1}^1 \Big(D_x^2 + H_x^2 + \frac{1}{2\varepsilon} \eta^0 D^2 +  \frac{1}{2\varepsilon} \eta D^2 \Big) dx - \tilde{\phi}_x (-1,t) \int_{-1}^1 (\eta^0 D + \delta^0 H) dx, \notag
\end{align}
and
\begin{align}
&\quad\ \frac{d}{dt} \bigg[ \frac{1}{2}\int_{-1}^1 (\bar{D}^2 + \bar{H}^2) dx + d^2 + h^2 \bigg] \notag\\
&= -\int_{-1}^1 \Big(\bar{D}_x^2 + \bar{H}_x^2 + \frac{1}{2\varepsilon} \eta^0 \bar{D}^2 + \frac{1}{2\varepsilon} \eta \bar{D}^2 \Big)dx - \frac{2m_0 \gamma_\varepsilon d^2}{(1+\gamma_\varepsilon)\varepsilon} \label{id74}\\
&\quad\quad - \frac{(1+2\gamma_\varepsilon) d}{(1+\gamma_\varepsilon)\varepsilon} \int_{-1}^1 \eta^0 \bar{D} dx - \frac{d}{\varepsilon} \int_{-1}^1 \bar{H}_x \bar{D} dx+ \frac{d}{(1+\gamma_\varepsilon)\varepsilon} \int_{-1}^1 \delta^0 \bar{H} dx, \notag
\end{align}
where $\eta = n+p$, $(n,p,\phi)$ is the corresponding solution of the PNP
system (\ref{id1})-(\ref{id2}) with initial data satisfying $n\left( x,0 \right)={{n}^{0}}\left( x \right)+\tilde{n}\left( x,0 \right)$, $p\left( x,0 \right)={{p}^{0}}\left( x \right)+\tilde{p}\left( x,0 \right)\ge 0$ for $x\in (-1,1)$.
\end{thm}
\begin{proof}
The proof of Theorem \ref{thm5} is similar to those of Theorem \ref{thm1} and Theorem \ref{thm2}, the difference is to deal with the nonlinear terms. By integrating the equations for $\tilde{\delta}$ and $\tilde{\eta}$
in (\ref{id12}) from $-1$ to $x$, we obtain
\begin{align}
D_t &= D_{xx} - \eta^0 \tilde{\phi}_x - \psi_x H_x - \tilde{\phi}_x H_x, \label{id75} \\
H_t &= H_{xx} - \delta^0 \tilde{\phi}_x - \psi_x D_x - \tilde{\phi}_x D_x. \label{id76}
\end{align}
Multiply (\ref{id75}) by $D$, integrate it from $-1$ to $1$, and do integration by parts, then
\begin{align}\label{id77}
\frac{1}{2} \frac{d}{dt} \int_{-1}^{1} D^2 dx = -\int_{-1}^1 \bigg(D_x^2 + \eta^0 \tilde{\phi}_x D + \psi_x H_x D + \tilde{\phi}_x H_x D\bigg) dx.
\end{align}
Here we have used the fact that $D(-1)=D(1)=0$. On the other hand, we multiply (\ref{id76}) by $H$, integrate it from $-1$ to $1$, and do integration by parts, then
\begin{align}\label{id78}
\frac{1}{2} \frac{d}{dt} \int_{-1}^{1} H^2 dx = -\int_{-1}^1 \bigg(H_x^2 + \delta^0 \tilde{\phi}_x H + \psi_x D_x H + \tilde{\phi}_x D_x H\bigg) dx.
\end{align}
Similarly, we have used the fact that $H(-1)=H(1)=0$. By the same argument in Theorem \ref{thm1}, we have
\begin{align}
&\quad\ \int_{-1}^1 \bigg(\eta^0 \tilde{\phi}_x D + \psi_x H_x D + \delta^0 \tilde{\phi}_x H + \psi_x D_x H\bigg) dx \label{id79}\\
&= \frac{1}{\varepsilon} \int_{-1}^1 \eta^0 D^2 dx + \tilde{\phi}_x (-1,t) \int_{-1}^1 (\eta^0 D + \delta^0 H) dx. \notag
\end{align}
For the nonlinear terms, we have
\begin{align}
\int_{-1}^1 \bigg(\tilde{\phi}_x H_x D + \tilde{\phi}_x D_x H\bigg) dx &= \int_{-1}^1 \tilde{\phi}_x (DH)_x dx \notag\\
&= -\int_{-1}^1 \tilde{\phi}_{xx} DH dx \notag\\
&= -\frac{1}{\varepsilon} \int_{-1}^1 \tilde{\delta} DH dx \label{id80}\\
&= -\frac{1}{\varepsilon} \int_{-1}^1 D_x DH dx \notag\\
&= \frac{1}{2\varepsilon} \int_{-1}^1 D^2 H_x dx \notag\\
&= \frac{1}{2\varepsilon} \int_{-1}^1 \tilde{\eta} D^2 dx. \notag
\end{align}
Notice that $\eta = \eta^0 + \tilde{\eta}$. Combining (\ref{id77})-(\ref{id80}), we obtain (\ref{id73}). Now we use the relations between $D$, $H$, $\bar{D}$, $\bar{H}$, $d$ and $h$, and (\ref{id73}). Then
\begin{align}
&\quad\ \frac{d}{dt} \bigg[ \frac{1}{2}\int_{-1}^1 (\bar{D}^2 + \bar{H}^2) dx + d^2 + h^2 \bigg] \notag\\
&= \frac{1}{2} \frac{d}{dt} \int_{-1}^1 (D^2 + H^2) dx \notag\\
&= -\int_{-1}^1 \Big(D_x^2 + H_x^2 + \frac{1}{2\varepsilon} \eta^0 D^2 +  \frac{1}{2\varepsilon} \eta D^2 \Big) dx - \tilde{\phi}_x (-1,t) \int_{-1}^1 (\eta^0 D + \delta^0 H) dx \label{id81}\\
&= -\int_{-1}^1 \Big(\bar{D}_x^2 + \bar{H}_x^2 + \frac{1}{2\varepsilon} \eta^0 (\bar{D}+d)^2 +  \frac{1}{2\varepsilon} \eta (\bar{D}+d)^2 \Big) dx \notag\\
&\quad\quad - \tilde{\phi}_x (-1,t) \int_{-1}^1 (\eta^0 (\bar{D}+d) + \delta^0 (\bar{H}+h)) dx. \notag
\end{align}
As for (\ref{id37}), we have
\begin{align}\label{id82}
\tilde{\phi}_x (-1,t) = -\frac{1}{2(1+\gamma_\varepsilon)\varepsilon} \int_{-1}^1 D dx = -\frac{d}{(1+\gamma_\varepsilon)\varepsilon}\,,
\end{align}
because the Poisson's equation of $\tilde{\phi}$ here is the same as that of the linearized problem. Then (\ref{id81}) becomes
\begin{align}
&\quad\ \frac{d}{dt} \bigg[ \frac{1}{2}\int_{-1}^1 (\bar{D}^2 + \bar{H}^2) dx + d^2 + h^2 \bigg] \notag\\
&= -\int_{-1}^1 \Big(\bar{D}_x^2 + \bar{H}_x^2 + \frac{1}{2\varepsilon} \eta^0 (\bar{D}+d)^2 +  \frac{1}{2\varepsilon} \eta (\bar{D}+d)^2 \Big) dx \notag\\
&\quad\quad + \frac{d}{(1+\gamma_\varepsilon)\varepsilon} \int_{-1}^1 (\eta^0 (\bar{D}+d) + \delta^0 (\bar{H}+h)) dx \label{id83}\\
&= -\int_{-1}^1 \Big(\bar{D}_x^2 + \bar{H}_x^2 + \frac{1}{2\varepsilon} \eta^0 \bar{D}^2 + \frac{1}{2\varepsilon} \eta \bar{D}^2 \Big)dx - \frac{2m_0 \gamma_\varepsilon d^2}{(1+\gamma_\varepsilon)\varepsilon} \notag\\
&\quad\quad - \frac{(1+2\gamma_\varepsilon) d}{(1+\gamma_\varepsilon)\varepsilon} \int_{-1}^1 \eta^0 \bar{D} dx - \frac{d}{\varepsilon} \int_{-1}^1 \bar{H}_x \bar{D} dx+ \frac{d}{(1+\gamma_\varepsilon)\varepsilon} \int_{-1}^1 \delta^0 \bar{H} dx. \notag
\end{align}
Here we have used $\tilde{\eta} = \bar{H}_x$. Then we complete the proof of Theorem~\ref{thm5}.
\end{proof}

\noindent {\bf Proof of Theorem~\ref{nl-stability-thm}}

In order to use (\ref{id74}) for the proof of Theorem~\ref{nl-stability-thm}, we need to estimate terms in the right-hand side of (\ref{id74}). For the term $\int_{-1}^1 \eta \bar{D}^2 dx \geq 0$, we need the nonnegative sign of $\eta=n+p$ which may come from the following result:
\begin{prop}\label{rmk_eta}
Let $(n,p,\phi)$ be the solution of (\ref{id1})-(\ref{id2}) with initial data ${{n}_{0}},{{p}_{0}}\in {{L}^{2}}\left( -1,1 \right)$ and $n_0, p_0\geq 0$ for $x\in (-1,1)$. Then $n,p\ge 0$ for $x\in \left( -1,1 \right),t>0$.
\end{prop}
\noindent The proof of Proposition~\ref{rmk_eta} is standard and is given in Appendix~II. A similar proof can be found in~\cite{bhn-na94}. Proposition~\ref{rmk_eta} implies $\eta \left( x,t \right)=n(x,t)+p(x,t)\ge 0$ for $x\in \left( -1,1 \right),t>0$. Here we assume initial data $n_0(x)=n\left( x,0 \right)={{n}^{0}}\left( x \right)+\tilde{n}\left( x,0 \right)\ge 0$ and $p_0(x)=p\left( x,0 \right)={{p}^{0}}\left( x \right)+\tilde{p}\left( x,0 \right)\ge 0$ for $x\in (-1,1)$.

Now we need to deal with the last three integrals in the right-hand side of (\ref{id74}). For the integral $\frac{(1+2\gamma_\varepsilon) d}{(1+\gamma_\varepsilon)\varepsilon} \int_{-1}^1 \eta^0 \bar{D} dx$, similar to (\ref{id49}), we have
\begin{align}\label{est-nl-1}
\left|\frac{(1+2\gamma_\varepsilon) d}{(1+\gamma_\varepsilon)\varepsilon} \int_{-1}^1 \eta^0 \bar{D} dx\right| \leq \frac{1}{2\varepsilon}\int_{-1}^1\left|\eta^0-m_0\right|\bar{D}^2dx + \frac{(1+2\gamma_\varepsilon)^2 d^2}{2(1+\gamma_\varepsilon)^2\varepsilon}\int_{-1}^1\left|\eta^0-m_0\right|dx.
\end{align}
As for (\ref{id49}), we use $ab \leq \frac{1}{2}a^2 + \frac{1}{2}b^2$, $\forall a,b\geq 0$, instead of $ab \leq a^2 + \frac{1}{4}b^2$, $\forall a,b\geq 0$. We also use (\ref{et0-m0-2}) to estimate $\int_{-1}^1\left|\eta^0-m_0\right|dx$ and the fact $\gamma_\varepsilon > 0$ to get
\begin{align}
\left|\frac{(1+2\gamma_\varepsilon) d}{(1+\gamma_\varepsilon)\varepsilon} \int_{-1}^1 \eta^0 \bar{D} dx\right| &\leq \frac{1}{2\varepsilon}\int_{-1}^1\left|\eta^0-m_0\right|\bar{D}^2dx + \frac{m_0 K_0\phi_0(1)(1+2\gamma_\varepsilon)^2 d^2}{(1+\gamma_\varepsilon)^2 M\sqrt{\varepsilon}} \notag\\
&\leq \frac{1}{2\varepsilon}\int_{-1}^1\left|\eta^0-m_0\right|\bar{D}^2dx + \frac{m_0 K_0\phi_0(1)(1+2\gamma_\varepsilon)^2 d^2}{(1+\gamma_\varepsilon) M\sqrt{\varepsilon}}. \label{est-nl-2}
\end{align}
Furthermore, we use (\ref{et0-m0-1}) to get
\begin{align}
&\quad\ -\frac{1}{2\varepsilon}\int_{-1}^1 \eta^0 \bar{D}^2 dx - \frac{(1+2\gamma_\varepsilon) d}{(1+\gamma_\varepsilon)\varepsilon} \int_{-1}^1 \eta^0 \bar{D} dx \notag \\
&\leq -\frac{1}{2\varepsilon}\int_{-1}^1 \eta^0 \bar{D}^2 dx + \left|\frac{(1+2\gamma_\varepsilon) d}{(1+\gamma_\varepsilon)\varepsilon} \int_{-1}^1 \eta^0 \bar{D} dx\right| \notag \\
&\leq -\frac{1}{2\varepsilon}\int_{-1}^1 \left[m_0 +\left(\eta^0 - m_0 - \left|\eta^0-m_0\right|\right)\right]\bar{D}^2dx + \frac{m_0 K_0\phi_0(1)(1+2\gamma_\varepsilon)^2 d^2}{(1+\gamma_\varepsilon) M\sqrt{\varepsilon}} \notag \\
&\leq -\frac{m_0}{2\varepsilon}\left[1 - K_0\phi_0(1)\frac{\sqrt{\varepsilon}}{M}\right]\int_{-1}^1 \bar{D}^2dx + \frac{m_0 K_0\phi_0(1)(1+2\gamma_\varepsilon)^2 d^2}{(1+\gamma_\varepsilon) M\sqrt{\varepsilon}}\,. \label{est-nl-3}
\end{align}
By (\ref{id52}), we have
\begin{align}
\bigg| \frac{d}{(1+\gamma_\varepsilon)\varepsilon} \int_{-1}^1 \delta^0 \bar{H} dx \bigg| \leq \frac{1}{2} \int_{-1}^1 \bar{H}_x^2 dx + \frac{3{{d}^{2}}}{2(1+{{\gamma }_{\varepsilon }})}{{m}_{0}}{{K}_{0}}{{\phi }_{0}}\left( 1 \right) \frac{1}{M\sqrt{\varepsilon}}\,, \label{est-nl-4}
\end{align}
for $0 < \varepsilon < \varepsilon_0$, where $\varepsilon_0$ comes from (\ref{est-grad-psi-2}) and depends only on $m_0$ and $\phi_0(1)$. H\"{o}lder's and Young's inequalities give
\begin{align}
\left|\frac{d}{\varepsilon} \int_{-1}^1 \bar{H}_x \bar{D} dx\right| &\leq \frac{d}{\varepsilon} \left(\int_{-1}^1 \bar{H}_x^2 dx\right)^{\frac{1}{2}} \left(\int_{-1}^1 \bar{D}^2 dx\right)^{\frac{1}{2}} \notag\\
&\leq \frac{m_0}{4\varepsilon}\int_{-1}^1 \bar{D}^2 dx + \frac{d^2}{m_0 \varepsilon}\int_{-1}^1 \bar{H}_x^2 dx. \label{est-nl-5}
\end{align}
Recall (\ref{id74}):
\begin{align*}
&\quad\ \frac{d}{dt} \bigg[ \frac{1}{2}\int_{-1}^1 (\bar{D}^2 + \bar{H}^2) dx + d^2 + h^2 \bigg] \\
&= -\int_{-1}^1 \Big(\bar{D}_x^2 + \bar{H}_x^2 + \frac{1}{2\varepsilon} \eta^0 \bar{D}^2 + \frac{1}{2\varepsilon} \eta \bar{D}^2 \Big)dx - \frac{2m_0 \gamma_\varepsilon d^2}{(1+\gamma_\varepsilon)\varepsilon} \\
&\quad\quad - \frac{(1+2\gamma_\varepsilon) d}{(1+\gamma_\varepsilon)\varepsilon} \int_{-1}^1 \eta^0 \bar{D} dx - \frac{d}{\varepsilon} \int_{-1}^1 \bar{H}_x \bar{D} dx+ \frac{d}{(1+\gamma_\varepsilon)\varepsilon} \int_{-1}^1 \delta^0 \bar{H} dx.
\end{align*}
Substituting (\ref{est-nl-3})--(\ref{est-nl-5}) into (\ref{id74}), we get
\begin{align}
&\quad\ \frac{d}{dt} \bigg[ \frac{1}{2}\int_{-1}^1 (\bar{D}^2 + \bar{H}^2) dx + d^2 + h^2 \bigg] \notag\\
&\leq -\int_{-1}^1 \bigg(\bar{D}_x^2 + \left(\frac{1}{2}-\frac{d^2}{m_0 \varepsilon}\right)\bar{H}_x^2\bigg)dx - \frac{m_0}{4\varepsilon}\left[1 - 2K_0\phi_0(1)\frac{\sqrt{\varepsilon}}{M}\right]\int_{-1}^1 \bar{D}^2dx \label{est-nl-6}\\
&\quad\ -\frac{2m_0 \gamma_\varepsilon d^2}{(1+\gamma_\varepsilon)\varepsilon} + \frac{m_0 K_0\phi_0(1)(1+2\gamma_\varepsilon)^2 d^2}{(1+\gamma_\varepsilon) M\sqrt{\varepsilon}} + \frac{3{{d}^{2}}}{2(1+{{\gamma }_{\varepsilon }})}{{m}_{0}}{{K}_{0}}{{\phi }_{0}}\left( 1 \right) \frac{1}{M\sqrt{\varepsilon}}. \notag
\end{align}
We set $1 - 2K_0\phi_0(1)\frac{\sqrt{\varepsilon}}{M} > \frac{1}{2}$ for $0 < \varepsilon < \varepsilon_2$, where $\varepsilon_2 > 0$ depending only on $m_0$ and $\phi_0(1)$. As for Theorem~\ref{prop1}, the last three terms of (\ref{est-nl-6}) become nonpositive if $\frac{\gamma_\varepsilon}{\sqrt{\varepsilon}} > \frac{\left((1+2\gamma_{max})^2+3\right)K_0\phi_0(1)}{4M}$ holds true. Hence we obtain (\ref{nl-stability}) by letting $\tilde{\varepsilon}' := \min\{\varepsilon_0,\varepsilon_2\}$.

Now we claim that if $I_0 < \frac{\theta m_0\varepsilon}{2}$ for some $0<\theta<1$, then $d^2(t) \leq I(t) \leq I(0) = I_0$ for all $t>0$, where $I(t) := \frac{1}{2}\int_{-1}^1 (\bar{D}^2 + \bar{H}^2) dx + d^2 + h^2$. Notice that by (\ref{nl-stability}),
\begin{align}\label{I-0-decrease}
\frac{d}{dt}\left[\frac{1}{2}\int_{-1}^1 (\bar{D}^2 + \bar{H}^2) dx + d^2 + h^2\right] \leq 0\,,
\end{align}
provided that $d^2 \leq I_0 \leq \frac{m_0\varepsilon}{2}$. Assume $I_0 < \frac{\theta m_0\varepsilon}{2}$ for some $0<\theta<1$. Then (\ref{I-0-decrease}) implies that $d^2(t) \leq I(t) \leq I(0)=I_0 < \frac{\theta m_0\varepsilon}{2}$ for all $t>0$ (see Appendix~III for the detail). Therefore, we conclude that $\frac{1}{2}-\frac{d^2}{m_0 \varepsilon} > \frac{1}{2}(1-\theta) > 0$ for all $t \geq 0$, and complete the proof of Theorem~\ref{nl-stability-thm}.

\noindent {\bf Proof of Corollary~\ref{cor4-1} }

\noindent By Theorem \ref{nl-stability-thm}, we have
\begin{align*}
\frac{d}{dt}\left[\frac{1}{2}\int_{-1}^{1}{({{{\bar{D}}}^{2}}+{{{\bar{H}}}^{2}})}dx+{{d}^{2}}+{{h}^{2}}\right] \leq -\int_{-1}^1 \bigg(\bar{D}_x^2 + \frac{1}{2}(1-\theta)\bar{H}_x^2\bigg)dx - \frac{m_0}{8\varepsilon} \int_{-1}^1 \bar{D}^2 dx
\end{align*}
for $t>0$ and $0<\varepsilon<\tilde{\varepsilon}'$. As for the proof of Corollary \ref{cor1}, we have
\begin{align*}
\frac{d}{dt} \bigg[ \frac{1}{2}\int_{-1}^1 (\bar{D}^2 + \bar{H}^2) dx + d^2 + h^2 \bigg] \leq -\alpha' \bigg[ \frac{1}{2}\int_{-1}^1 (\bar{D}^2 + \bar{H}^2) dx + d^2 + h^2 \bigg]
\end{align*}
for some $\alpha' > 0$ depending only on $\theta$ and the constants from the Poincar\'{e}'s inequality. Therefore, we complete the proof of Corollary~\ref{cor4-1}.

\section*{Appendix~I}
\ \ \ \
Here we want to prove that ${{\left\| D \right\|}_{L_{x}^{2}}}+{{\left\| H \right\|}_{L_{x}^{2}}}$ is equivalent to ${{\left\| {\tilde{n}} \right\|}_{H_{x}^{-1}}}+{{\left\| {\tilde{p}} \right\|}_{H_{x}^{-1}}}$, which means that
$$
{{C}_{1}}\left( {{\left\| D \right\|}_{L_{x}^{2}}}+{{\left\| H \right\|}_{L_{x}^{2}}} \right)\le {{\left\| {\tilde{n}} \right\|}_{H_{x}^{-1}}}+{{\left\| {\tilde{p}} \right\|}_{H_{x}^{-1}}}\le {{C}_{2}}\left( {{\left\| D \right\|}_{L_{x}^{2}}}+{{\left\| H \right\|}_{L_{x}^{2}}} \right)
$$
for some constants $C_j>0$, $j=1,2$ independent of $D, H, \tilde{n}$ and $\tilde{p}$. Because $\tilde{n}=\frac{1}{2}(\tilde{\delta}+\tilde{\eta})$ and $\tilde{p}=\frac{1}{2}(\tilde{\delta}-\tilde{\eta})$, it is sufficient to show that
\BE\label{equi-nm}
{{C}_{1}}{{\| D \|}_{L_{x}^{2}}}\le {{\| {\tilde{\delta }} \|}_{H_{x}^{-1}}}\le {{C}_{2}}{{\| D \|}_{L_{x}^{2}}}\quad\hbox{and}\quad {{C}_{1}}{{\| H \|}_{L_{x}^{2}}}\le {{\| {\tilde{\eta }} \|}_{H_{x}^{-1}}}\le {{C}_{2}}{{\| H \|}_{L_{x}^{2}}}\,,
\EE
for some constants $C_j>0$, $j=1,2$ independent of $D, H, \tilde{\delta}$ and $\tilde{\eta}$. For any fixed $t>0$, we set $\tilde{\delta}\in H^{-1}_x\left((-1,1)\right)$ and denote $\langle \tilde{\delta},v \rangle$ and $\|\tilde{\delta}\|_{H^{-1}_x}$ as follows:
\begin{align*}
\langle \tilde{\delta},v \rangle = \int_{-1}^1\, \tilde{\delta}v dx \quad\mbox{ for all }\: v\in H^1\left((-1,1)\right)\,,
\end{align*}
and
\begin{align*}
\|\tilde{\delta}\|_{H^{-1}_x} = \sup_{\|v\|_{H^1_x}=1} \int_{-1}^1 \tilde{\delta}\, v\, dx,
\end{align*}
where $\langle\ ,\ \rangle$ denotes the pairing between $H^{-1}$ and $H^1$. By the definition of $D$ (see (\ref{id21})), we have $D_x = \tilde{\delta}$, $D(\pm 1,t) = 0$ and hence $\langle \tilde{\delta},v \rangle=-\int_{-1}^1\,D\,v_x\,dx$ for $v\in H^1((-1,1))$ using integration by part.

Now we claim that $\|\tilde{\delta}\|_{H^{-1}_x}$ is equivalent to $\|D\|_{L^2_x}$, which means that ${{C}_{1}}{{\| D \|}_{L_{x}^{2}}}\le {{\| {\tilde{\delta }} \|}_{H_{x}^{-1}}}\le {{C}_{2}}{{\| D \|}_{L_{x}^{2}}}$ for some constants $C_j>0, j=1,2$ independent of $D$ and $\tilde{\delta }$. For any $v\in H^1((-1,1))$ with $\|v\|_{H^1_x}=1$,
\begin{align*}
\left|\langle \tilde{\delta},v \rangle\right| = \left|\int_{-1}^1 Dv_x dx \right| \leq \|D\|_{L^2_x}\|v_x\|_{L^2_x} \leq \|D\|_{L^2_x}\,.
\end{align*}
Here we have used Holder's inequality and the fact that $1=\left\| v \right\|_{{{H}^{1}}}^{2}=\left\| v \right\|_{{{L}^{2}}}^{2}+\left\| {{v}_{x}} \right\|_{{{L}^{2}}}^{2}$. Consequently, $\|\tilde{\delta}\|_{H^{-1}} \leq \|D\|_{L^2}$. On the other hand, let
\begin{align*}
\mathcal{D}(x,t) := \int_{-1}^x D(s,t)ds - \frac{1}{2}\int_{-1}^1\int_{-1}^x D(s,t)dsdx\,,\quad\hbox{for}\: x\in (-1,1)\,, \: t>0\,.
\end{align*}
Then $\mathcal{D}_x = D$, $\int_{-1}^1 \mathcal{D}dx = 0$, and hence we have
\begin{align*}
\int_{-1}^1 D^2 dx &= -\int_{-1}^1 \,\mathcal{D}\,\tilde{\delta}\, dx \\
&\leq \|\tilde{\delta}\|_{H^{-1}_x}\|\mathcal{D}\|_{H^1_x} \\
&\leq C\|\tilde{\delta}\|_{H^{-1}_x}\|D\|_{L^2_x} \mbox{ by Holder's and Poincar\'{e}'s inequalities}\,,
\end{align*}
which implies $\|D\|_{L^2_x} \leq C \|\tilde{\delta}\|_{H^{-1}_x}$. Note that $D^2=\mathcal{D}_x\,D$, $D_x = \tilde{\delta}$, $D(\pm 1,t) = 0$, and here we have used integration by parts for the first equality. Therefore, $\|\tilde{\delta}\|_{H^{-1}_x}$ is equivalent to $\|D\|_{L^2_x}$. Similarly, we may get the equivalence between $\|\tilde{\eta}\|_{H^{-1}_x}$ and $\|H\|_{L^2_x}$ and complete the proof of (\ref{equi-nm}).

\section*{Appendix~II}
\ \ \ \
Here we state the proof of Proposition~\ref{rmk_eta}. Multiply equations of $n$ and $p$ of (\ref{id1}) by $n_- := \min\{n,0\}$
and $p_- := \min\{p,0\}$, respectively, and integrate them over $(-1,1)$. Then
\begin{align*}
\frac{1}{2}\frac{d}{dt}\int_{-1}^1 n_-^2 dx &= -\int_{-1}^1 \left((n_-)_x^2 - n_- (n_-)_x\phi_x\right) dx, \\
\frac{1}{2}\frac{d}{dt}\int_{-1}^1 p_-^2 dx &= -\int_{-1}^1 \left((p_-)_x^2 + p_-(p_-)_x\phi_x\right) dx.
\end{align*}
We use the interpolation inequality $\|u\|_{L^3}\leq
C\|u\|_{L^2}^{1/2}\|u\|_{H^1}^{1/2}$ for $u\in H^1(-1,1)$ to get
\begin{align*}
\bigg|\int_{-1}^1 n_- (n_-)_x\,\phi_x\, dx \bigg| &\leq \int_{-1}^1 |n_- (n_-)_x\,\phi_x| dx \\
&\leq \|(n_-)_x\|_{L^2_x((-1,1))} \|n_-\|_{L^3_x((-1,1))} \|\phi_x\|_{L^6_x((-1,1))} \\
&\leq C_1 \|(n_-)_x\|_{L^2_x((-1,1))} \|n_-\|_{L^2_x((-1,1))}^{1/2} \|n_-\|_{H^1_x((-1,1))}^{1/2} \|\phi_x\|_{L^6_x((-1,1))} \\
&\leq C_1 \|n_-\|_{H^1_x((-1,1))}^{3/2} \|n_-\|_{L^2_x((-1,1))}^{1/2} \|\phi_x\|_{L^6_x((-1,1))} \\
&\leq \frac{1}{2} \|n_-\|_{H^1_x((-1,1))}^2 + C_2 \|n_-\|_{L^2_x((-1,1))}^2 \|\phi_x\|_{L^6_x((-1,1))}^4,
\end{align*}
where $C_1$, $C_2$ depend only on the domain $(-1,1)$. Using Poisson's equation of $\phi$ in (\ref{id1}) with Robin boundary condition of (\ref{id2}),
we have $\|\phi_x\|_{L^6_x((-1,1))} \leq
C_3(1+\|n-p\|_{L^2_x((-1,1))})$, where $C_3$ depends only on the domain $(-1,1)$,
$\varepsilon$, $\gamma_\varepsilon$, and $\phi_0(\pm 1)$. Therefore,
\begin{align}
\frac{1}{2}\frac{d}{dt}\int_{-1}^1 n_-^2 dx &\leq C_2 C_3 \left(1+\|n-p\|_{L^2_x((-1,1))}\right)^4 \int_{-1}^1 n_-^2 dx \label{eta_n}\\
&= f(t) \int_{-1}^1 n_-^2 dx, \notag
\end{align}
where $f(t) := C_2 C_3 \left(1+\|n-p\|_{L^2_x((-1,1))}\right)^4 \in L^1((0,T))$ for
any $T>0$ since $n,p \in L^\infty(0,T;L^2(-1,1))$ solve (\ref{id1})-(\ref{id2}). By (\ref{eta_n}) and the
fact that $n_-(x,0) = 0$ for $x\in (-1,1)$, we have $n_- \equiv 0$, i.e., $n \geq 0$.
Similarly, we get
\begin{align}
\frac{1}{2}\frac{d}{dt}\int_{-1}^1 p_-^2 dx &\leq f(t) \int_{-1}^1 p_-^2 dx, \label{eta_p}
\end{align}
and $p_-(x,0) = 0$. Therefore, $p_- \equiv 0$, i.e., $p \geq 0$ and we may complete the proof of Proposition~\ref{rmk_eta}.

\section*{Appendix~III}
\ \ \ \
Here we prove that $I\left( t \right)\le I\left( 0 \right)={{I}_{0}}<\frac{\theta {{m}_{0}}\varepsilon }{2}$ for $t>0$ if $I\left( 0 \right)={{I}_{0}}<\frac{\theta {{m}_{0}}\varepsilon }{2}$ holds true for some $0<\theta <1$. Using (\ref {I-0-decrease}), it is equivalent to show that $\Im =\left\{ t\ge 0:I\left(s\right)\le \frac{\theta {{m}_{0}}\varepsilon }{2} \mbox{ for } 0 \le s \le t\right\}=\left[ 0,\infty  \right)$ if $I\left( 0 \right)={{I}_{0}}<\frac{\theta {{m}_{0}}\varepsilon }{2}$ holds true for some $0<\theta <1$. By the continuity of function $I$, there exists ${{t}_{1}}>0$ such that $I\left( t \right)<\frac{{{m}_{0}}\varepsilon }{2}$ for $0\le t \le {{t}_{1}}$, which satisfies the condition of (\ref {I-0-decrease}). Please note that ${{d}^{2}}\left( t \right)\le I\left( t \right)$ for $t\ge 0$. Consequently, (\ref {I-0-decrease}) implies that $I\left( {{t}_{1}} \right)\le I\left( t \right)\le I\left( 0 \right)<\frac{\theta {{m}_{0}}\varepsilon }{2}$ for $0\le t \le {{t}_{1}}$. That is, $t_1 \in \Im$. Moreover, we claim that $\Im $ is open in $\left[ 0,\infty  \right)$. Suppose $0<{{t}_{2}}\in \Im $. Then $I\left(t\right)\le \frac{\theta {{m}_{0}}\varepsilon }{2}$ for all $t \le t_2$, which implies $[0,t_2] \subset \Im$. By the continuity of function $I$, there exists $\delta >0$ such that $I\left( t \right)<\frac{{{m}_{0}}\varepsilon }{2}$ for ${{t}_{2}}\le t<{{t}_{2}}+\delta $, and the condition of (\ref {I-0-decrease}) holds true for ${{t}_{2}}\le t<{{t}_{2}}+\delta $. Hence by (\ref {I-0-decrease}), $I\left( {{t}_{2}}+\delta  \right)\le I\left( t \right)\le I\left( {{t}_{2}} \right)\le \frac{\theta {{m}_{0}}\varepsilon }{2}$ for ${{t}_{2}}\le t<{{t}_{2}}+\delta $, which implies that $\left[ {{t}_{2}},{{t}_{2}}+\delta  \right)\subset \Im $ and $\Im $ is open in $\left[ 0,\infty  \right)$. On the other hand, it is trivial that $\Im $ is closed in $\left[ 0,\infty  \right)$ because of the continuity of function $I$. Therefore, $\Im =\left\{ t\ge 0:I\left(s\right)\le \frac{\theta {{m}_{0}}\varepsilon }{2} \mbox{ for all } 0 \le s \le t\right\}=\left[ 0,\infty  \right)$ and (\ref {I-0-decrease}) gives $I\left( t \right)\le I\left( 0 \right)={{I}_{0}}<\frac{\theta {{m}_{0}}\varepsilon }{2}$ for $t>0$.

\section{Acknowledgment}
\ \ \ \ Chia-Yu Hsieh wishes to express sincere thanks to the Department of Mathematics of Pennsylvania State University for the chance of one-year visit.
Tai-Chia Lin is partially supported by the National Science Council of Taiwan grants NSC-102-2115-M-002-015 and NSC-100-2115-M-002-007.


\begin{thebibliography}{10}

\bibitem{AMT-ttsp00} {\sc A. Arnold, P. Markowich, G. Toscani}, {\em On large time asymptotics for drift-diffusion Poisson systems}, Transport Theory Statist. Phys. 29(2000), no. 3-5, 571-581.

\bibitem{BCE} {\sc V. Barcilon, D. P. Chen, R. S. Eisenberg, and J. W.
Jerome}, {\em Qualitative properties of steady-state Poisson-Nernst-Planck systems: perturbation and simulation study}, SIAM J. APPL. MATH. Vol.57, No.3, pp.631--648 (1997).

\bibitem{bhn-na94} P. Biler, W. Hebisch and T. Nadzieja, The Debye system: existence and large time behavior of solutions, Nonlinear Analysis, TMA, Vol.23, No.9, pp.~1189-1209, 1994.

\bibitem{BD-ahp00} {\sc P. Biler and J. Dolbeault}, {\em Long Time Behavior of Solutions to Nernst-Planck and Debye-Huckel Drift-Diffusion Systems}, Ann. Henri Poincare, 1 (2000) 461-472.

\bibitem{CLB-bp97} {\sc D Chen, J Lear, and B Eisenberg} {\em Permeation through an open channel: Poisson-Nernst-Planck theory of a synthetic ionic channel}, Biophys J. (1997), 72(1) pp. 97-116.

\bibitem{E98} {\sc B. Eisenberg}, {\em Ionic Channels in Biological Membranes: Natural Nanotubes}, Acc. Chem. Res., 31 (1998), pp.117--123.

\bibitem{EL-sm07}{\sc B. Eisenberg and W. Liu}, {\em Poisson-Nernst-Planck systems for ion channels with permanent charges}, SIAM J. Math. Anal. 38-6(2007), pp. 1932-1966.

\bibitem{G-zamm85}{\sc H. Gajewski} {\em On existence, uniqueness and asymptotic behavior of solutions of the basic equations for carrier transport in semiconductors}, Z. Angew. Math. Mech. 65(1985), no. 2, 101-108.

\bibitem{H1} {\sc B. Hille}, {\em Ion channels of excitable membranes}, 3rd Edition, Sinauer Associates, Inc. (2001).

\bibitem{H-Z-81} {\sc R. J. Hunter}: {\em Zeta Potential in Colloid Science}, Academic Press Inc. (1981).

\bibitem{LMB} {\sc D. Lacoste, G.I. Menon, M.Z. Bazant, and J.F.
Joanny}, {\em Electrostatic and electrokinetic contributions to the elastic moduli of a driven membrane}, Eur. Phys. J. E 28 (2009) 243--264.

\bibitem{LLHLL} {\sc C. C. Lee, H. Lee, Y. Hyon, T. C. Lin and C. Liu},
{\em New Poisson-Boltzmann Type Equations: One-Dimensional Solutions}, Nonlinearity 24 (2011) 431--458.

\bibitem{Lu} {\sc W. Liu}, {\em Geometric singular perturbation approach
to steady-state Poisson-Nernst-Planck systems}, SIAM J. Appl. Math. (2005) Vol.65. No.3, pp.754--766.

\bibitem{MRS} {\sc P. A. Markowich, C. A. Ringhofer, C. Schmeiser}, {\em Semiconductor equations}, Springer-Verlag, Vienna, (1990).

\bibitem{Mori1} {\sc Y. Mori, J.W. Jerome, and C.S. Peskin},
{\em A Three-dimensional Model of Cellular Electrical Activity},
Bulletin of the Institute of Mathematics Academia Sinica, 2(2)
(2007), pp.~367--390.

\bibitem{PJ} {\sc J. H. Park and J. W. Jerome}, {\em Qualitative properties of steady-state Poisson-Nernst-Planck systems: mathematical study}, SIAM J. APPL. MATH. Vol.57, No.3, pp.609--630 (1997).

\bibitem{RCA89} {\sc O. J. Riveros, T. L. Croxton, and W. M. Armstrong},
{\em Liquid Junction Potentials Calculated From Numerical Solutions of the Nernst-Planck and Poisson Equations}, J. Theor. Biol., 140(1989), pp.221--230.

\bibitem{Rolf2} {\sc R. Ryham, C. Liu and L. Zikatanov},
{\em Mathematical Models for the Deformation of Electrolyte
Droplets}, Discrete Contin. Dyn. Syst. Ser. B 8 (2007), no.~3,
p.~649--661.

\bibitem{WXLLS-N-14} {\sc L. Wan, S. Xu, M. Liao, C. Liu and P. Sheng}, {\em New Perspectives on Electrokinetics}, Phys. Rev. X (2014), in press.

\end{thebibliography}
\end{document}